\documentclass[reqno]{amsart}
\usepackage{array}   
\usepackage{bm}
\usepackage{cite}
\usepackage{enumerate}
\usepackage{graphicx}
\usepackage{geometry}  
\usepackage{hyperref}
\usepackage{lineno}

\usepackage{stmaryrd}
\usepackage{subfigure}
\usepackage{subcaption}
\SetSymbolFont{stmry}{bold}{U}{stmry}{m}{n}
\usepackage{tabu}
\usepackage{tabularx}
\usepackage{tikz} 
\usepackage{url} 
\usepackage[normalem]{ulem} % delete line
\usepackage{enumitem}

\usepackage{amsmath,amssymb,amsthm,amsfonts }
\newtheorem{theorem}{Theorem}[section]
\newtheorem{lemma}{Lemma}[section]

\newtheorem{remark}{Remark}[section]
\numberwithin{equation}{section}

\newcommand{\energy}[1]{\interleave#1\interleave_{E}}
\newcommand{\ener}[1]{\interleave#1\interleave}
\newcommand{\norm}[1]{\left\Vert#1\right\Vert}
\newcommand{\abs}[1]{\left\vert#1\right\vert}
\newcommand{\bra}{\langle}
\newcommand{\ket}{\rangle}
\newcommand{\inn}[1]{\langle #1 \rangle}

\newcommand{\Bup}{\boldsymbol{\Upsilon}}
\newcommand{\E}{\boldsymbol{E}}

\newcommand{\f}{\boldsymbol{f}}
\newcommand{\g}{\boldsymbol{g}}
\newcommand{\bG}{\boldsymbol{G}}
\newcommand{\fp}{{\boldsymbol{f}^\prime}}
\newcommand{\fpp}{{\boldsymbol{f}^{\prime\prime}}}
\newcommand{\gp}{{\boldsymbol{g}^\prime}}

\newcommand{\Csol}{ C_{\rm sol} }
\newcommand{\D}{\mathrm{D}}

\newcommand{\ii}{\boldsymbol{ \mathrm{i} }}

\newcommand{\V}{\boldsymbol{V}}

\newcommand{\Vh}{\boldsymbol{V}_h}

\newcommand{\bO}{\mathcal{O}}

\newcommand{\Th}{\mathcal{M}_h}

\newcommand{\Eht}{{\E}_{h,T}}
\newcommand{\Eh}{{\E}_h}

\newcommand{\PEp}{\boldsymbol{P}^{+}_h\E }

\newcommand{\Ph}{\boldsymbol{P}_h }

\newcommand{\Pih}{{\boldsymbol{\Pi}_h} }

\newcommand{\bph}{{\boldsymbol{\phi}}}

\newcommand{\ps}{{\boldsymbol{\Psi}}}
\newcommand{\pps}{{\boldsymbol{P}^{-}_h\boldsymbol{\Psi}}}

\newcommand{\w}{{\boldsymbol{w}}}

\newcommand{\bL}{{\boldsymbol{L}}}
\newcommand{\bH}{{\boldsymbol{H}}}
\newcommand{\tbH}{{\boldsymbol{\widetilde H}}}

\newcommand{\bX}{{\boldsymbol{X}}}

\newcommand{\Psie}{{\ps}_{\mathcal{E}}}
\newcommand{\Psia}{{\ps}_{\mathcal{A}}}

\newcommand{\Ee}{{\E}_{\mathcal{E}}}
\newcommand{\Ea}{{\E}_{\mathcal{A}}}
\newcommand{\gE}{{\g}_{\mathcal{E}}}
\newcommand{\gA}{{\g}_{\mathcal{A}}}
\newcommand{\GE}{{\bG}_{\mathcal{E}}}
\newcommand{\GA}{{\bG}_{\mathcal{A}}}

\newcommand{\bv}{\boldsymbol{v}}

\newcommand{\bu}{{\boldsymbol{u}}}

\newcommand{\bz}{{\boldsymbol{z}}}
\newcommand{\bw}{{\boldsymbol{w}}}
\newcommand{\vh}{{\boldsymbol{v}_h}}

\newcommand{\bbeta}{{\boldsymbol{\eta} }}
\newcommand{\bxi}{{\boldsymbol{\xi} }}

\newcommand{\cf}{\rm{cf. }}

\newcommand{\eq}[1]{\begin{align}#1\end{align}}
\newcommand{\eqn}[1]{\begin{align*}#1\end{align*}}

\newcommand{\fa}{\mathsf{f}}

\newcommand{\ls}{\lesssim}

\newcommand{\De}{\Delta}
\newcommand{\n}{\boldsymbol{\nu}}

\newcommand{\Ga}{\Gamma}
\newcommand{\la}{\lambda}

\newcommand{\na}{\nabla}

\newcommand{\Om}{\Omega}
\newcommand{\pa}{\partial}

\newcommand{\si}{\sigma}

\newcommand{\vp}{\varphi}
\newcommand{\ka}{\kappa}

\newcommand{\N}{\mathbb{N}}
\newcommand{\R}{\mathbb{R}}
\newcommand{\C}{\mathbb{C}}

\newcommand{\Sm}{\mathcal{T}}
\newcommand{\cs}{C_\mathcal{T}}
\newcommand{\ctr}{C_{\rm tr}}
\newcommand{\ext}{\mathcal{E}_{\Om}}
\DeclareMathOperator{\re}{{Re}}
\DeclareMathOperator{\im}{{Im}}
\DeclareMathOperator{\curl}{\boldsymbol{\mathrm{curl}}}
\DeclareMathOperator{\dive}{{\mathrm{div}}}
\DeclareMathOperator{\diam}{{diam}}

\newcommand{\qaq}{\quad\mbox{and}\quad}

\graphicspath{{figures/} }
\begin{document}
	
	\title[Higher-order EEM for time-harmonic Maxwell Equation]{Preasymptotic error estimates of 
		higher-order EEM \\  for the  time-harmonic Maxwell equations\\ with large wave number}
	\markboth{S. Lu and H. Wu}{error estimates of higher-order EEM for Maxwell equations}
	
	\author[S. Lu]{Shuaishuai Lu}
	\address{School of Mathematics, Nanjing University, Jiangsu, 210093, P.R. China. }
	\curraddr{}
	\email{ssl@smail.nju.edu.cn}
	\thanks{This work was partially supported  by National Key R\&D Program of China 2024YFA1012600, and by the NSF of China under grants 12171238 and 12261160361.}
	
	\author[H. Wu]{Haijun Wu}
	\address{School of Mathematics, Nanjing University, Jiangsu, 210093, P.R. China. }
	\curraddr{}
	\email{hjw@nju.edu.cn}
	\thanks{}
	
	
	\date{}
	
	\dedicatory{}
	
	
	\begin{abstract}

		The time-harmonic Maxwell equations with impedance boundary condition and large wave number are discretized using the second-type N\'{e}d\'{e}lec's edge element method (EEM). Preasymptotic error bounds are derived, showing that, under the mesh condition $\kappa^{2p+1}h^{2p}$ being sufficiently small, the error of the EEM of order $p$ in the energy norm is bounded by $\mathcal{O}\big(\kappa^{p}h^p + \kappa^{2p+1}h^{2p}\big)$, while the error in the $\kappa$-scaled $\boldsymbol{L}^2$ norm is bounded by $\mathcal{O}\big((\kappa h)^{p+1} + \kappa^{2p+1} h^{2p}\big)$. Here, $\kappa$ is the wave number and $h$ is the mesh size. Numerical tests are provided to  illustrate our theoretical results.
	\end{abstract}
	
	\maketitle
	
	\section{Introduction}\label{sc1}
	The numerical solution of Maxwell equations remains a significant challenge in scientific computing, given their fundamental role in modeling electromagnetic phenomena. These equations are vital to numerous applications, including electromagnetic scattering, optical systems, and modern communication technologies.  The time-harmonic Maxwell equations, in particular, arise from the frequency domain analysis of electromagnetic waves (cf. \cite{monk2003,jin2015theorycem}).
	
	In this work, we focus on the time-harmonic Maxwell equations governing the electric field $\E$, which is subject to the standard impedance boundary condition: 
	\begin{subequations} 
		\label{Maxwell}
		\begin{align} 
			\curl \curl \E - \ka^2 \E = \f & \quad \text{in } \Om, \label{eq:eq}  \\
			\curl \E \times \n - \ii\ka\la\E_T = \g & \quad \text{on } \Ga := \pa \Om, \label{eq:bc}
		\end{align} 
	\end{subequations} 
	where $\Om \subset \R^{3}$ is a bounded Lipschitz domain. In these equations, $\ii = \sqrt{-1}$ is the imaginary unit, $\n$ is the unit outward normal to $\Ga$, and $\E_{T} := (\n \times \E) \times \n$ represents the tangential component of $\E$ on $\Ga$. The wave number $\ka \gg 1$ and the impedance constant $\la > 0$ are two given physical parameters (cf. \cite{monk2003}). The right-hand side $\f$ corresponds to a prescribed current density, which is divergence-free (i.e. $\dive \f =0$). It is assumed that $\g \cdot \n = 0$, leading to $\g_{T} = \g$. The problem \eqref{Maxwell} describes the propagation of electromagnetic waves at a single frequency in scattering scenarios, especially when  the boundary of the computational domain is neither a perfect conductor nor free space. 
	
	A variety of numerical approaches have been developed for solving electromagnetic problem \eqref{eq:eq}, including boundary element methods  \cite{buffa2002boundary,buffa2003galerkin,buffa2003boundary}, discontinuous Galerkin (DG) methods \cite{feng2014absolutely,feng_lu_xu_2016,hiptmair2013error,lu2017HDGMaxwell,chen2019hybridizable}, and certain interior penalty methods \cite{perugia2002,houston2005interior,pplu2019}. Among these, the finite element method (FEM), particularly N\'{e}d\'{e}lec's $ \bH(\curl; \Om)$-conforming edge element method (EEM) \cite{Nedelec1980,Nedelec1986}  is regarded as one of the most prominent and effective techniques in computational electromagnetics, as emphasized in \cite{nedelec2001acoustic,monk1992finite,monk1993analysis,monk2003simple,monk2003,Nedelec1980,Nedelec1986,houston2005interior}.
	
	A major challenge when solving high-frequency wave scattering problems  is that FEM of fixed order may suffer from the ``pollution effect'', where accuracy deteriorates as the wave number increases \cite{Babuska1997pollution,ihlenburg1995finite,ihlenburg1997finite}. This issue is particularly significant in both theoretical and practical applications, necessitating the derivation of error estimates that explicitly account for the impact of $\ka$. In what follows, by ``asymptotic error estimate", we mean a wave-number-explicit error estimate without pollution error (i.e., quasi-optimal); in contrast, by ``preasymptotic error estimate", we mean one with nonnegligible pollution effect. The understanding of FEM for the high-frequency Helmholtz problem $\Delta u + \ka ^2 u = f$ has seen significant advancements over the past few decades through various works  \cite{Babuska1997pollution,ihlenburg1995finite,ihlenburg1997finite,melenk2010dtn,melenk2011,ZhuWu2013II,DuWu2015,Wu2014Pre,Wu2018jssx,LiWupml2019,galkowski2024sharp,chaumont2019general,li2023higherorderpml}. For example, it has been shown in \cite{DuWu2015,galkowski2024sharp} that when applying FEM  to the Helmholtz equation with an impedance boundary condition, the following preasymptotic error bound holds if $\ka^{2p+1}h^{2p}$ is sufficiently small:
	\eq{\label{eq:helmBound}
		\ener{u-u_h}_h  \ls   (1+ \ka^{p+1}h^{p}  ) \inf_{v_{h} \in V_h}\ener{u-v_h}_h \ls \bO\big(\ka^p h^p+ \ka^{2p+1} h^{2p}\big),
	} 
	where  $\ener{\cdot}_h=\big(\norm{\nabla \cdot}^2+\ka^2\|\cdot\|^2\big)^{\frac{1}{2}}$, $V_h$ is the $p$-th order Lagrange finite element space  and $u_h$ is the discrete solution. The above error bound \eqref{eq:helmBound} shows that the actual error of FEM is the sum of the best approximation error plus a pollution term which grows rapidly when $\ka \rightarrow +\infty$.  On one hand, preasymptotic error estimates like \eqref{eq:helmBound} provide valuable insights into how the wave number  $\ka$ influence the reliability and accuracy of FEM as $\ka $ increases. And one can deduce that the quasi-optimality of $p$-th order FEM  is ensured by the condition that $\ka^{p+1}h^p$ is sufficiently small.  
	On the other hand, the above mesh condition is practical, as achieving a discrete solution with reasonable accuracy requires the pollution term $\bO(\ka^{2p+1} h^{2p})$ to be small.  Both works \cite{DuWu2015} and \cite{galkowski2024sharp} are based on the  idea of modified duality argument \cite{ZhuWu2013II}  to obtain preasymptotic error estimates \eqref{eq:helmBound}. However, they employ different techniques to obtain results under the weakest currently known mesh condition  that $ \ka^{2p+1} h^{2p}$ is sufficiently small. The work in \cite{DuWu2015} uses the  traditional elliptic-projection combined with the  negative-norm estimates technique, while the authors of \cite{galkowski2024sharp} adapt the  traditional elliptic-projection   by constructing a novel elliptic operator, which can be obtained by adding an appropriate smoothing operator to the original partial differential operator.  
	
	Now, a natural and important question arises: how does the accuracy deteriorate as $\ka \rightarrow +\infty$ when the EEM is applied to the time-harmonic Maxwell equations? Or more precisely, does  a  similar preasymptotic error bound \eqref{eq:helmBound} still hold for the  Maxwell cases? Compared to high-frequency Helmholtz-type problems, analyses for indefinite Maxwell problems are more intricate because of the nontrivial kernel space of the  $\curl$ operator. Up to now, the $\ka$-implicit convergence analysis (see, e.g.,  \cite{monk2003,gatica2012finite,monk2003simple}) and asymptotic error estimates of EEM for \eqref{eq:eq} are well-established (see, e.g.,  \cite{melenk2020wavenumber,melenk2023wavenumber,chaumont2023asymptotic}) and existing dispersion analysis on structured meshes (see, e.g., \cite{ainsworth2003dispersive,Ainsworth2004,ainsworth2004dispersive}) suggest that the EEM applied to the Maxwell problems has similar dispersive behavior as that of the FEM  for the Helmholtz problem. Among these works, a   milestone is the asymptotic error estimates proposed in \cite{melenk2020wavenumber,melenk2023wavenumber} for the time-harmonic Maxwell equations with transparent and impedance boundary
	conditions, respectively. They {proved} that the discrete solution of $hp$-EEM is pollution-free (i.e., quasi-optimal)  
	under the condition that $p \geq C \ln \ka$ for some $C>0$ %$p \geq \max\{1, C \ln \ka\}$ for any $C>0$ 
	and $\ka h/p$ is sufficiently small.   In contrast, for the fixed order EEM, for instance, in the case of impedance boundary condition, it requires that  $\ka^{p}h^{p-1}$ is sufficiently small (see \cite[Lemma~9.5]{melenk2023wavenumber}, note that $p=2$ here corresponds to $p=1$ in that reference, which refers to the quadratic edge element of the first type).   
	%	For the fixed order EEM (note that $p=0$ corresponds to the lowest edge elements in \cite{melenk2020wavenumber,melenk2023wavenumber}), it requires that $\ka^{p+7}h^p$ is sufficiently small (see \cite[Remark~4.19]{melenk2020wavenumber}) for the case of transparent boundary condition or $\ka^{p+1}h^p$ is sufficiently small (see \cite[Lemma~9.5]{melenk2023wavenumber}) for the case of impedance boundary condition. 
	However,  rigorous preasymptotic error estimates for EEM applied to time-harmonic Maxwell equations remain incomplete in the literature.
	Addressing the above question forms the primary motivation of this paper.
	
	Research on preasymptotic error estimates of EEM applied to \eqref{eq:eq}  has only recently seen some progress. In our recent work  \cite{lu2024preasymptotic},  we  developed preasymptotic error estimates for \eqref{Maxwell} using linear EEM and some $\bH(\curl)$-conforming interior penalty EEM.  It was proved that if $\ka^3 h^2\leq C_0$ for some $C_0>0$, then
	\eq{
		\ener{\E-\Eh} &\ls   \bO \big(\ka h+ \ka^3 h^2\big),
	} 
	where $\ener{\cdot}=\big(\norm{\curl \cdot}^2+\ka^2\|\cdot\|^2 + \ka\la\|\cdot_T\|^2 _{ \Ga }\big)^{\frac{1}{2}}$ is the energy norm. Instead of deriving preasymptotic error estimates for EEM by following the usual approach (see, e.g., \cite{monk2003}),  the core of  \cite{lu2024preasymptotic} lies in a  wave-number-explicit stability estimate of the discrete solution $\Eh$ by using the modified duality argument, which is then used to derive the desired error estimate. 
	
	More recently, for the heterogeneous time-harmonic Maxwell equations with perfect electric conductor (PEC) boundary conditions (i.e., $ \E\times \n   ={\bm{0}}$ on $\Ga$), Chaumont-Frelet and co-authors  \cite{chaumont2024sharp} established the preasymptotic error bound 
	\eq{ \ener{\E-\E_h }  \leq C_1 \Big(1 + (\ka h )^p \ka^2 \Csol \Big) \min_{\bv_h \in \V_{h}}  \ener{\E-\bv_h },	
		\label{eq:H1bound}}
	under the mesh condition $\ka^{2p+2}h^{2p} \Csol \leq C_2$
	for some positive constants  $C_1$ and $C_2$. 
	Here, $\Csol$ is the $\bL^2(\Om)\to \bL^2(\Om)$ norm of the solution operator $\f \mapsto \E$, and $\V_{h}$ represents the first- or second-type  edge element space of order $p$. This work extends  the methodology proposed in \cite{galkowski2024sharp} to the time-harmonic Maxwell problems with PEC boundary conditions.  A key insight from \cite{chaumont2024sharp} is that the analogous elliptic operator in\cite{galkowski2024sharp} for Maxwell problems can be constructed by adding a smoothing operator for functions lying in the complement of the kernel space of the $\curl$ operator.

	In this paper, a sequel of our previous work \cite{lu2024preasymptotic}, we study the higher-order EEM  for the time-harmonic Maxwell  problem \eqref{Maxwell} with impedance boundary condition  and prove the following preasymptotic error  estimate  under the condition that  $\ka^{2p+1} h^{2p}$ is sufficiently small:
	\eq{\ener{\E-\Eh} &\ls   \big( 1 +  \ka^{p+1}h^{p} \big) \inf_{\bv_{h} \in \V_h}\ener{\E -\bv_h}\ls \bO\big(\ka^p h^p+ \ka^{2p+1} h^{2p}\big), \label{eq:introPreErr}}
	where $\Vh$ is the second-type edge element space of order $p$ and $\Eh\in \Vh $ is the EEM solution. 
	Note that the methodology presented in \cite{chaumont2024sharp}  for the Maxwell problem with PEC boundary condition cannot be directly  applied to the analysis of problem \eqref{Maxwell}. For example,  Assumption~2.3~(\textrm{ii}) in \cite{chaumont2024sharp} relies on the fact that functions in $\bH_{0}(\curl;\Om)$ have vanishing tangential trace, and the regularity results established  in \cite[Lemma~11.2, Lemma~11.5]{chaumont2024sharp} do not hold for  system  \eqref{Maxwell}. To  overcome  this limitation and develop new preasymptotic error estimates for higher-order EEM for the problem \eqref{Maxwell}, we first drew inspiration from the  auxiliary elliptic  operators  used  in \cite{galkowski2024sharp,li2023higherorderpml,chaumont2024sharp}, and constructed a new Maxwell elliptic   operator tailored for  the Maxwell problem \eqref{Maxwell} with impedance boundary conditions, using a smoothing operator defined simply by truncating eigenfunction expansions. We then provide a new and simpler proof of the regularity decomposition result from \cite[Theorem 7.3]{melenk2023wavenumber} but under a weaker smoothness assumption on the domain $\Omega$. Based on this decomposition result and using the trick of ``modified duality argument" \cite{ZhuWu2013II,lu2024preasymptotic,chaumont2024sharp,li2023higherorderpml}, we may prove that the EEM satisfies the same stability estimates (see Theorem~\ref{thm:EEMStability}) as the continuous problem \eqref{Maxwell} (see \eqref{eq:stability0}), under the mesh condition that  $\ka^{2p+1} h^{2p}$ is sufficiently small. Finally, we derive the preasymptotic error estimates for the EEM in both energy and $\bL^2$ norms (see \eqref{eq:introPreErr} and Theorem~\ref{thm:ErrorResult}). 
	Moreover, our preasymptotic error bounds do not contain any $\ka$-dependent constants.
	The key contributions of this work are summarized as follows:
	\begin{itemize}
		\item  Wave-number-explicit $\bH^{m}$ regularity  and regularity by decomposition  results  (see Theorems~\ref{thm:stability} and \ref{thm:reg-dec} for the problem  \eqref{Maxwell})  are derived under the assumption of a $C^{m}$-domain. These results improve upon those in \cite{melenk2023wavenumber,lu2018regularity}.
		\item  The preasymptotic error estimate  \eqref{eq:introPreErr} in the energy norm is proved under the mesh condition that $\ka^{2p+1}h^{2p}$ is sufficiently small. Using the second-type edge element method, this result generalizes the well-known preasymptotic error bound \eqref{eq:helmBound}  for the Helmholtz problem to the more challenging  time-harmonic Maxwell equations with impedance boundary condition. 
		\item  Preasymptotic error estimate in the $\bL^2$ norm is also obtained under the same mesh condition.  This is the first such result for higher-order EEM of the second type applied to time harmonic Maxwell problems. 
	\end{itemize}
	
	The remainder of this paper is organized as follows. In Section \ref{sc2}, we initially present the variational formulation of \eqref{Maxwell} and formulate the EEM. Subsequently, we furnish preliminary results related to stability, regularity, and approximation estimates. Section \ref{sc3} introduces the auxiliary Maxwell elliptic problem and proves error estimates for the corresponding elliptic projections.   In Section \ref{sc4}, we  prove the decomposition result and derive the preasymptotic stability and  error estimates for the EEM. Finally, in Section~\ref{sc5}, we conduct  numerical experiments to illustrate our theoretical findings.
	
	Throughout this paper, we use notations $A \ls B$ and $A \gtrsim B$ for
	the inequalities $A \leq C B$ and $A \geq C B$ respectively,
	where $C$ is a positive number independent of the mesh size and wave number $\ka$,
	but the value of which may vary in different instances. 
	The shorthand notation $A \simeq B$   represents  $A \leq C B$ and $B \leq C A$ hold simultaneously. 
	We suppose $\la \simeq 1$ and  $\ka \gg 1$ since we  consider high frequency problem.   Let $\mathbb{N} =\{0,1,2,\cdots\}$ and $\N_1 = \N\setminus \{0\}$.
	In this paper, we assume   that the bounded open set $\Om \in \R^3$
	is strictly star-shaped with respect to a point ${\bm x}_0\in\Om$, that is, there exists $C>0$ such that 
	$({\bm x}-{\bm x}_0)\cdot \n \geq C \;\; \forall\bm x\in\Ga$. We  remark that in this setting, $\Ga=\pa \Om$ is connected. Without loss of generality, we assume that $\diam(\Om)\simeq 1$ since  {the other case of $\diam(\Om)\not\simeq 1$ can be transformed to this case} through a suitable coordinate scaling.

	\section{EEM and preliminaries}\label{sc2}
	In this section, we first introduce some Sobolev spaces and  formulate the variational formulation and EEM for our model problem \eqref{Maxwell}. We then recall some preliminary results which will be useful in the later sections.
	
	\subsection{Spaces and norms}
	To state the variational formulation of \eqref{Maxwell}, we first introduce some notations.  For a bounded domain $D\subset \R^3$ with a Lipschitz 
	boundary $\pa D$,  and any real number $s \ge 0$,  we shall use the standard Sobolev spaces   {$L^{\infty}(D)$ and  $H^s(D)$, their norms $\|\cdot\|_{L^{\infty}(D)}$ and  $\|\cdot\|_{s,D}$}, and
	refer to \cite{adams2003sobolev,brenner2008mathematical,monk2003} for their definitions.  
	On the boundary of $D$, the space $H^s(\pa D)$ is well-defined if $D$ is $C^m$ for some $m \in \N_1$ and $s < m$ (cf., \cite{amrouche1998vector}).  
	For vector-valued functions, these spaces are indicated by boldface symbols, such as  $\bH^s(D)$ and $\bH^s(\pa D)$.
	When  $D=\Om$, we  {omit the subscript $D$} and the index $s$ if $s=0$ and use the shorthands   {$ \|\cdot\| := \|\cdot\|_{0,\Om},  \|\cdot\|_{\Ga}:= \|\cdot\|_{0,\Ga}$} for simplicity.
	We denote by $(\cdot, \cdot)_{D}$ and $\bra\cdot, \cdot \ket_{\Sigma}$ for
	$\Sigma \subset \pa D$
	the $\bL^2$ inner-product on $D$ and  {$\pa \Sigma$, respectively. For simplicity, write} $(\cdot, \cdot):=(\cdot, \cdot)_{\Om}$ and $\bra\cdot, \cdot \ket:=\bra\cdot, \cdot \ket_{\Ga}$.
	%	$\bL^2$-inner product (conjugate-linear in the second argument) on complex-valued $\bL^2(D)$ and $\bL^2(\Sigma)$ spaces, respectively.  
	We introduce the following $\ka$-weighted norms $\norm{\cdot}_{m,\ka}$  on $\bH^m(\Om)$ for $m \in \N$ and $\norm{\cdot}_{m-\frac12,\ka,\Ga}$  on  $\bH^{m-\frac12}(\Ga)$ for $m \in \N_1$, respectively:
	\eq{\norm{\cdot}_{m,\ka}:=\bigg(\sum_{j=0}^m\ka^{2j}\|\cdot\|_{m-j}^2\bigg)^\frac12,\quad
		\norm{\cdot}_{m-\frac12,\ka,\Ga}:=\bigg(\sum_{j=0}^{m-1}\ka^{2j}\|\cdot\|_{m-j-\frac12,\Ga}^2\bigg)^\frac12,}
	and define the following negative norm on  $\tbH^{-m}(\Om):=\big(\bH^m(\Om)\big)'$:
	\eqn{\norm{\cdot}_{-m,\ka}:=\sup_{{\bm 0}\neq \bv\in \bH^m(\Om)}\frac{|\bra\cdot, \bv\ket_{\tbH^{-m}\times \bH^m}|}{\norm{\bv}_{m,\ka}}.}

	\begingroup
	\allowdisplaybreaks	
	We introduce some spaces on the domain $\Om$
	\begin{align*}
		H_0^1(\Om) &:= \{v \in H^1(\Om) :  v|_{\Ga} =0\},\\ 
		\bm{H}(\curl ;\Om)&:=\{\bm{v}\in \bL^2(\Om): \curl \bm{v}\in \bL^2(\Om) \},\\
		\bm{H}^s(\curl ;\Om)&:=\{\bm{v}\in \bH^s(\Om): \curl \bm{v}\in \bH^s(\Om) \},\\
		\bm{H}_0(\curl ;\Om)&:=\{\bm{v}\in \bm{H}(\curl ;\Om):   \bm{v} \times \n|_{\Ga}=\bm{0} \},\\
		\bm{H}(\dive  ;\Om)&:=\{\bm{v}\in \bL^2(\Om):  \dive \bm{v}\in L^2(\Om) \},\\
		%		\bm{H}^s(\dive ;\Om)&:=\{\bm{v}\in \bH^s(\Om): \dive \bm{v}\in      H^s(\Om) \},\\
		\bm{H}_0(\dive  ;\Om)&:=\{\bm{v}\in \bm{H}(\dive  ;\Om): \bm{v}\cdot\n|_{\Ga}=0 \},\\
		\bm{H}(\dive  ^0;\Om)&:=\{\bm{v}\in \bL^2(\Om): \dive \bm{v}=0\}, \\
		\bm{H}^s(\dive^0 ;\Om)&:=\{\bm{v}\in \bH^s(\Om): \dive \bv=0 \},\\
		\bm{H}_0(\dive  ^0;\Om)&:= \bm{H}_0(\dive  ;\Om) \cap \bm{H}(\dive  ^0;\Om), \\
		\bX &:=\bm{H}(\curl ;\Om)\cap \bm{H}_0(\dive^0 ;\Om),
	\end{align*}
	\endgroup
	and some spaces on $\Ga$ (cf. \cite{buffa2002traces})
	\eqn{ 
		\bm{H}^s(\dive_{\Ga} ;\Ga)&:=\{\bm{v}\in \bH^s(\Ga): \dive_{\Ga} \bm{v}\in H^s(\Ga) \}, \\
		\bL^2_t(\Ga) &:=\{\bm{v}\in \bL^2(\Ga): \bv\cdot\n |_{\Ga} = 0 \}, }
	with norms
	\eqn{
		\norm{\bm v}_{\bH(\curl)}&=\big(\norm{\curl\bm v}^2+\norm{\bm v}^2\big)^\frac12, \\
		\norm{\bm v}_{\bH^s(\curl)}&=\big(\norm{\curl\bm v}_s^2+\norm{\bm v}_s^2\big)^\frac12,
	}
	where $v|_{\Ga}$ is understood as the trace from $H^1(\Om)$ to $H^\frac12(\Ga)$ and $\bv \times \n|_\Ga$  
	is the tangential trace from $\bH(\curl; \Om )$ to $\bH_{\dive}^{-1/2}(\Ga)$. Similarly, we will understand 
	$\bv_T$ as the tangential trace from $\bH(\curl; \Om )$ to $\bH_{\curl}^{-1/2}(\Ga)$  (see, e.g., \cite{monk2003,melenk2020wavenumber}), $\bv\cdot\n|_{\Ga}$ is the normal trace from $\bH(\dive;\Om)$ onto $\bH^{-\frac{1}{2}}(\Ga)$, and	$\dive_\Ga$ is the  surface divergence operator (see, e.g., \cite{monk2003,melenk2023wavenumber}). We also denote by $\bm{H}^s_t(\dive_{\Ga} ;\Ga):=\bm{H}^s(\dive_{\Ga} ;\Ga)\cap \bL^2_t(\Ga)$
	%	 and $\bm{H}^s_t( \Ga ):= \bm{H}^s( \Ga )\cap \bL^2_t(\Ga)$ 
	for simplicity.
	\subsection{ The EEM}
	First, we introduce the variational formulation for the Maxwell problem \eqref{Maxwell}. Define the ``energy" space
	\eqn{&\V:= \big\{ \bv\in \bH(\curl; \Om) :  {\bv_T \in \bL^2(\Ga) \big\}}  \quad\text{with norm}\quad \ener{\bv}:=\big(\norm{\curl\bm v}^2+\ka^2\|\bv\|^2 +\ka\la\|\bv_T\|^2 _{ \Ga }\big)^{\frac{1}{2}},}
	and the following sesquilinear form   $a : \V\times \V \rightarrow \C$ by
	\begin{equation}
		a\left({\bm u}, {\bm v} \right):=\left(\curl {\bm u}, \curl {\bm v}\right) 
		-\ka^{2}\left({\bm u}, {\bm v} \right)-\ii \ka \la \inn{ {\bm u}_T,{\bm v}_T }.
		\label{eq:a}
	\end{equation}
	Then the variational formulation of \eqref{Maxwell} reads as:
	Find $\E  \in \boldsymbol{V} $ such that
	\begin{equation}
		a \left(\E, {\bm v} \right)=\left(\boldsymbol{f}, 
		{\bm v} \right)+\left\langle \g,{{\bm v}_T}\right\rangle 
		\quad \forall {\bm v} \in {\boldsymbol{V}}.
		\label{eq:dvp}
	\end{equation}
	The  well-posedness of the above problem is guaranteed by \cite[Theorem~4.17]{monk2003}.

	Next we follow \cite{melenk2023wavenumber} to introduce a regular and quasi-uniform triangulation  $\mathcal{M}_{h}$  of $\Om$, which satisfies:
	\begin{enumerate}
		\item {The (closed)  elements} $K\in \Th$ cover $\Om$, i.e.,  
		{$\bar{\Om}=\cup_{K\in\Th}K$.}  
		\item Each element $K$ is  associated with  a $C^1$-diffeomorphism 
		$F_K:\hat{K}\rightarrow K$. The set $\hat{K}$ is the reference tetrahedron. Denoting $h_K:=\diam(K)$, there holds, with some shape-regularity constant $\si$,
		\eqn{h_K^{-1}\norm{\D F_K }_{L^\infty(\hat K)}+h_K\norm{(\D F_K )^{-1}}_{L^\infty(\hat K)}\le \si,}
		where $\D F_K $ is the Jacobian matrix of $F_K$. 
		\item The intersection of two elements is only empty, a vertex, an edge, a face, or they coincide
		(here, vertices, edges, and faces are the images of the corresponding
		entities on the reference element $\hat{K}$). The parameterization of common edges
		or faces is compatible. That is, if two elements $K$, $K'$ share an edge (i.e.,  
		$F_K(e) = F_{K'}(e')$ for edges $ e, e'$ of $\hat{K}$ ) or a face 
		(i.e., $F_K ( {\fa} ) = F_{K'} ( \fa' )$ for faces $\fa, \fa'$ of $\hat{K}$), 
		then $F^{-1}_K \circ  F_{K'}: \fa' \rightarrow \fa $ is an affine isomorphism.
	\end{enumerate}
	To ensure the approximability of higher-order EEM (see Lemma~\ref{lem:interpEst}), we also assume that the curved triangulation is regular in the following sense (cf. \cite[Assumption~10.1]{chaumont2024sharp}): the element maps $F_K$ satisfies
	\eqn{
		h_K^{-|\alpha|}\norm{ \D^{\alpha}F_K }_{L^\infty(\hat K)}  \leq C  
		\quad \forall K\in \Th, \;  1 \leq \abs{\alpha}\leq p+2,
	}
	where $\alpha =(\alpha_1,\alpha_2,\alpha_3) \in \N^3$ is a multi-index with length $\abs{\alpha} = \sum_{i=1}^3 \alpha_i$.
	In the following, we denote by $h:=\max_{K\in\Th} h_K$. 
	
	Next, we introduce the EEM. Let $\V_{h}$ 
	be the $p$-th ($p\in \N_1$) order N\'{e}d\'{e}lec edge element space of the second type (cf. \cite{monk2003}):
	\begin{equation}
		\V_{h}:=\Big\{\bv_{h} \in \V :\; (\D F_K  )^T  ( \left.\bv_{h}\right|_{K} )\circ F_K
		\in\big(\mathcal{P}_{p}(\hat{K})\big)^{3}  \;\forall K \in \mathcal{M}_{h}\Big\}.	
		\label{eq:FEMspace}
	\end{equation}
	Here, $\mathcal{P}_{p}(\hat{K})$ denotes the set of polynomials of degree less than or equal to $p$ on $\hat{K}$.
	The EEM for the Maxwell problem \eqref{Maxwell} reads as: %  (see \eg \cite{monk2003}):
	Find $ \Eh \in \boldsymbol{V}_{h}$ such that
	\begin{equation}
		a\left(\Eh, {\bm v}_h\right)=\left(\boldsymbol{f}, 
		{\bm v}_h\right)+ \inn{\g,{\bm v}_{h,T} } 
		\quad \forall {\bm v}_h \in \boldsymbol{V}_{h}.
		\label{eq:EEM}
	\end{equation}
	It is clear that this is a consistent discretization formulation which means that if
	$\E \in$ $\V$, then
	\begin{equation}
		a \left(\E-\Eh,{\bm v}_h \right)=0 \quad
		\forall {\bm v}_h \in \V_h.	
		\label{eq:orthogonality}
	\end{equation}

	\subsection{Some approximation results}
	For any $\bv \in \bH^1(\curl;\Om)$, denote by $\Pih \bv \in \V_{h}$ its $p$-th order edge element interpolant. 
	We have the following interpolation error estimates.
	
	\begin{lemma}\label{lem:interpEst}
		For any  $  j \in \N_1 $, $j\leq p$, and $\bv \in \bH^{j+1}(\Om)$, there hold that
		\eq{   \|\curl(\bv - \Pih \bv )\| & \ls h^j  \|\bv\|_{\bH^{j}(\curl)}, \label{eq:curlinterp}  \\
			h^{\frac{1}{2}} \|(\bv - \Pih \bv)_T\|_{\Ga} +	\|\bv - \Pih \bv\|  & \ls h^{j+1} \|\bv\|_{j+1}. \label{eq:interp} 	  	
		}
		Furthermore, if $kh \ls 1$, we have 
		\eq{
			\ener{ \bv - \Pih \bv } & \ls h^{j} \|\bv\|_{j+1}.  \label{eq:interpener}
		}
	\end{lemma} 
	\begin{proof}
		The higher-order first-type N\'{e}d\'{e}lec interpolation error estimates on curved domains have been proved in \cite[Theorem~10.2]{chaumont2024sharp}. To prove \eqref{eq:curlinterp}  and the $\bL^2$ error estimate  in \eqref{eq:interp}  for the second-type N\'{e}d\'{e}lec interpolation, it is sufficient   to substitute the first-type N\'{e}d\'{e}lec interpolation operator $\widehat{I^c}$ in 	\cite[Appendix~B]{chaumont2024sharp} with the second-type counterpart and follow the same proof process. The boundary estimate follows from the
		local trace inequality and \eqref{eq:interpener} follows from \eqref{eq:curlinterp}--\eqref{eq:interp}, we omit the details.
	\end{proof}

	Let $\V_{h}^{0}:=\V_{h} \cap \bH_{0}(\curl; \Om ) $. The following lemma says that any discrete function in $\V_h$ has an ``approximation" in $\V_h^0$, whose error can be bounded by its tangential components on $\Ga$. The proof can be found in \cite[Proposition~4.5]{houston2005interior}  or \cite[Lemma 3.3]{lu2024preasymptotic}, and is omitted here.
	\begin{lemma}\label{lem:Vh0}
		For any $\bm v_h\in V_h$, there exists $\bm v_h^0\in \V_h^0$, such that
		\eq{\label{eq:Vh0}\|\bm v_h-\bm v_h^0\| \ls h^{\frac{1}{2}}\|\bm v_{h,T}\| _{ \Ga }\qaq \|\curl(\bm v_h-\bm v_h^0)\| \ls h^{-\frac{1}{2}}\|\bm v_{h,T}\| _{ \Ga }.}
	\end{lemma}

	\subsection{Helmholtz decompositions and embedding results}
	The following well-known decomposition of $\bL^2{(\Om)}$ will be frequently used in our subsequent analyses, the proof of which can be found in   \cite[Proposition~3.7.3]{assous2018mathematical} or       \cite[Theorem~4.2]{schweizer2018friedrichs}.
	\begin{lemma}[Helmholtz decomposition of $\bL^2(\Om)$]\label{lem:proju}
		Let $\Om$ be a bounded, open and connected subset of $\R^3$ with a Lipschitz
		boundary.  For any $ \bu \in \bL^2(\Om)$, there exists the following $\bL^2$-orthogonal decomposition:
		\eq{\label{eq:proju}
			\bu = \textsf{P}^0	\bu + \textsf{P}^{\perp} \bu, 
		}
		where $\textsf{P}^0:\bL^2(\Om)\rightarrow \nabla H^1(\Om)$ and $\textsf{P}^{\perp} : \bL^2(\Om) \rightarrow   \bm{H}_0(\dive^0;\Om) $ are projections.
	\end{lemma}

	\begin{remark}\label{rem:proj} 
		It is easy to see that for any $\bu \in \V$, there holds $\textsf{P}^0 \bu \in \V$ and $\textsf{P}^{\perp} \bu \in \bX \subset \V$. 		
	\end{remark}
	
	Let $U_{h}:=\left\{u \in H^{1}(\Om):
	(u|_{K})\circ F_K \in \mathcal{P}_{p+1}(\hat{K})\;  \forall K \in \mathcal{M}_{h}\right\}$ be the Lagrange finite element space of order $p+1$ and let
	$U_{h}^{0}:=U_{h} \cap H_{0}^{1}(\Om)$.
	It is easy to see that  $\nabla U_{h}^{0} \subset\V_{h}^{0} $ (cf. \cite{monk2003}). The following lemma gives both the Helmholtz decomposition and discrete
	Helmholtz decomposition for each $\bv_{h} \in \V_{h}^{0}$.
	\begin{lemma}[Helmholtz decompositions of $ {\V_{h}^{0}}$]\label{lem:HD}~\\
		For any $\vh \in  \V_{h}^{0}$, there exist
		$r^{0} \in H_{0}^{1}(\Om)$, $r_{h}^{0} \in U_{h}^{0}$, $\w^{0} \in \bH_{0}(\curl, \Om)$ 
		and $\w_{h}^0 \in \V_{h}^{0}$, such that
		\begin{align}
			&\vh = \nabla r_{h}^{0}+\w_{h}^0 =\nabla r^{0}+\w^{0},   \label{HD1}   \\
			& \left(\w_{h}^0, 
			\nabla \psi_{h}\right)=0 \quad \forall \psi_{h} \in U_{h}^{0},\quad \dive \w^{0} =0 \text { in } \Om,  \label{HD1w} \\
			&\left\|\w^{0}-\w_{h}^{0}\right\|  \ls   h\left\|\curl  \bv_{h}\right\|.
			\label{HD1approx} 
		\end{align}
	\end{lemma}
	\begin{proof}
		For the proof of this lemma, we refer to \cite[Theorem~3.45]{monk2003} and
		\cite[Lemma~7.6]{monk2003}.
		
	\end{proof}

	Next we recall the following embedding results (see, e.g.,  \cite{amrouche1998vector,girault1979finite}).
	\begin{lemma}\label{lem:embdding} Let $\Om$ be a bounded, open and connected subset of $\R^3$ with a $C^2$
		boundary.  Then $\bH_0(\curl ; \Om) \cap  \bH(\dive ; \Om)$
		and $\bH(\curl ; \Om) \cap  \bH_0(\dive ; \Om)$ are continuously embedded in $\bH^1(\Om)$ and 
		\eqn{ \|\bv\|_1 \ls \big(\|\curl \bv\| + \|\dive \bv\|\big)  \quad \forall \bv \in  \bH_0(\curl ; \Om) \cap  \bH(\dive ; \Om)  \text{ or }   \bH(\curl ; \Om) \cap  \bH_0(\dive ; \Om).
		}
	\end{lemma}
	We deduce by Lemma~\ref{lem:embdding} that $\bX \subset \bH^1{(\Om)}$ and  $\|\curl \cdot \|$ is a norm on $\bX$, i.e.
	%		Suppose that $\Om$ is convex or $C^2$, then $\bX$ is continuously embedded in $\bH^1(\Om)$, and
	\eq{ \|\bv\|_{1} \ls \|\curl \bv\|   \quad \forall \bv \in \bX.   \label{eq:emb1}
	}
	In addition, from the trace inequality $\norm{\bv}_\Ga\le \ctr\norm{\bv}^\frac12\norm{\bv}_1^\frac12$ (see e.g. \cite{brenner2008mathematical}) we have 
	\eq{\label{eq:trace} \norm{\bv}_\Ga\le \ctr\norm{\bv}^\frac12\|\curl \bv\|^\frac12   \quad \forall \bv \in \bX,}
	where the constant $\ctr$ depends only on $\Om$.

	\subsection{Stability and regularity results}  
	
	\begin{theorem}\label{thm:stability} Suppose $\Om$ is $C^{2}$. 
		Let $\E$ be the solution {to} the problem \eqref{Maxwell}. Assume that  $\f \in \bH (\dive^0;\Om)$ and $\g\in \bL^2_t(\Ga)$.
		Then we have 
		\eq{
			&	 	\|\curl  \E \| + \ka \|\E \|  +\ka \|\E_T \|_{\Ga} 
			\ls  \|\f\| +\|\g\|_{\Ga},   \label{eq:stability0}  \\
			&	 	\|\curl  \E \|_{1} +\ka\|\E\|_1
			\ls \ka\left(\|\f\| +\|\g\|_{\Ga}\right)+\|\g\|_{\frac{1}{2}, \Ga}. \label{eq:stability1}	 
		}
		Furthermore, if $\Om$ is $C^{m}$, 
		$\f \in \bH^{m  -2}(\dive^0;\Om)$, $\f\cdot\n \in \bH^{m-\frac{3}{2}}(\Ga)$ and $\g \in \bH^{m-\frac{3}{2}}_t(\dive_{\Ga};\Ga)$ 
		for some integer $m  \geq 2$, then 
		\eq{
			&   \|\E \|_{m}   \ls \ka^{m-1}  C_{m-2,\f,\g},    \label{eq:Hmstab} 
		}
		where 
		\eqn{C_{m-2,\f,\g}   &:= \|\f\|+\|\g\|_{\Ga} +  \ka^{-m+1}\|\f\|_{m-2,\ka} 
			+\ka^{-m}\|\f\cdot\n\|_{m-\frac{3}{2},\Ga}\\  
			&\quad+  
			\ka^{-m+1}\|\g\|_{m-\frac{3}{2},\ka, \Ga} 
			+ \ka^{-m}\|\dive_\Ga\g\|_{m- \frac{3}{2},  \Ga }.}
		
	\end{theorem}    
	\begin{proof}
		For the proofs of \eqref{eq:stability0} and  \eqref{eq:stability1} we refer to   \cite[Theorem~4.1]{chaumont2023explicit}, \cite[Theorem~3.1 and Remark~4.6]{hiptmair2011stability}.  
		Now we  prove \eqref{eq:Hmstab} by induction.  By \cite[Lemma~3.1]{lu2024preasymptotic}, \eqref{eq:Hmstab} holds when $m=2$.
		Next we suppose that
		\eqn{
			\|\E \|_{l}   \ls \ka^{l-1}  C_{l-2,\f,\g},  \quad 2\leq l\leq m-1.
		}
		Recall that $\dive_\Ga(\bm v\times\n)=\n\cdot\curl\bm v$ (see, e.g., \cite[(3.52)]{monk2003}). We have from \eqref{eq:bc} and \eqref{eq:eq} that
		\eqn{
			\dive_\Ga\big(\ii\ka\la\E_T+\g\big)&=\dive_\Ga(\curl \E \times  \n)=(\curl\curl \E)\cdot\n=(\ka^2 \E+\f)\cdot\n \quad \text{ on } \Ga,} which implies that
		\eq{
			\dive_\Ga(\ii\la\E_T)=\ka \E\cdot\n+\ka^{-1}\big(\f\cdot\n-\dive_\Ga\g\big). \label{eq:divEt}}
		Note that we can rewrite  \eqref{Maxwell} as
		\begin{alignat*}{2} 
			\curl \curl\E- \E  &= \f+(\ka^2-1)\E   \qquad &&{\rm in }\  \Om,    \\
			\curl\E \times  \n -\ii\la\E_T &=\g +\ii(\ka-1)\la\E_T  \qquad  &&{\rm on }\ \Ga.
		\end{alignat*} 
		Therefore, from \cite[Remark~3]{chen2024regularity} and \eqref{eq:divEt}, we have 
		\eqn{
			\|\E\|_m&\ls \big\|\f+(\ka^2-1)\E\big\|_{m-2}+\big\|\g +\ii(\ka-1)\la\E_T\big\|_{m-\frac{3}{2},\Ga}\\
			&\quad+\big\|(\f+(\ka^2-1)\E)\cdot\n-\dive_\Ga\g -\ii(\ka-1)\dive_\Ga(\la \E_T)\big\|_{m-\frac{3}{2},\Ga}\\
			&=\big\|\f+(\ka^2-1)\E\big\|_{m-2}+\big\|\g +\ii(\ka-1)\la\E_T\big\|_{m-\frac{3}{2},\Ga}\\
			&\quad +\big\|\ka^{-1}(\f\cdot\n-\dive_\Ga\g)+(\ka-1)\E\cdot\n\big\|_{m-\frac{3}{2},\Ga}\\
			&\ls \|\f\|_{m-2}+\ka^2\|\E\|_{m-2}+\|\g\|_{m-\frac{3}{2},\Ga}+\ka\|\E\|_{m-1}+\ka^{-1}\|\f\cdot\n\|_{m-\frac{3}{2},\Ga}+\ka^{-1}\|\dive_\Ga\g\|_{m-\frac{3}{2},  \Ga } \\\
			& \ls \ka^{m-1} \big( \|\f\|+\|\g\|_{\Ga}  \big) + \|\f\|_{m-2,\ka} +    \|\g\|_{m-\frac32,\ka, \Ga}  +\ka^{-1}\|\f\cdot\n\|_{m-\frac{3}{2},\Ga} + \ka^{-1}\|\dive_\Ga\g\|_{m - \frac{3}{2}, \Ga}  \\
			& =\ka^{m-1}  C_{m-2,\f,\g},
		}
		where we have used the trace inequality $\|\bv\|_{i-\frac{1}{2},\Ga}\ls \|\bv\|_{i}$ and the fact that $\|\dive_\Ga\bv\|_{i- \frac{1}{2},\Ga} \ls \|\bv\|_{i+ \frac{1}{2},\Ga}$. 
		%Then \eqref{eq:Hmstab} follows by combining \eqref{eq:stability0}--\eqref{eq:stability1} and the above estimate. 
		This completes the proof of the theorem.
	\end{proof}	
	
	\begin{remark}\label{rem:reg}
		{\rm (a)} The regularity estimate \eqref{eq:Hmstab} was proved in 
		\cite[Theorem 3.4]{lu2018regularity}  (and in \cite{melenk2025regularityvectorfieldspiecewise} for heterogeneous media) when the domain $\Om$ is $C^{m+1}$ and in \cite[Lemma~5.1]{melenk2023wavenumber} when $\Om$ is sufficiently smooth.  \cite[Theorem~1]{chen2024regularity} gave a regularity estimate by assuming  only $C^{m}$-domain,  	however, explicit dependence on the wave number $\ka$ was not considered there.
		
		{\rm (b)} 	If $\dive\f\neq 0$, we can obtain the estimates of $\E$ as follows.
		Inspired by the proof of \cite[Proposition~3.6]{melenk2023wavenumber}, let $\vp\in H_0^1(\Om)$ satisfy
		\eqn{-\De\vp=\dive\f.}
		By the elliptic regularity theory, the following estimates hold if $\Om$ is $C^{m+1}$:
		\eqn{  \|\nabla \vp\|_{m} \ls \|\dive \f\|_{m-1}, \quad \|\nabla \vp\cdot \n\|_{m-\frac{3}{2},\Ga}\ls \|\vp\|_{m} \ls \|\dive \f\|_{m-2}.
		}
		Let $\bm u=\E-\ka^{-2}\na\vp$. Clearly, $\bm u$ satisfies
		\begin{alignat}{2} 
			\curl \curl \bm u- \ka^2 \bm u &= \f+\na\vp   \qquad &&{\rm in }\  \Om,    \\
			\curl \bm u \times  \n -\ii\ka\la\bm u_T &=\g  \qquad &&{\rm on }\  \Ga.
		\end{alignat} 
		We have $\dive(\f+\na\vp)=0$ and $\dive\bm u=0$. Therefore, we can apply Theorem~\ref{thm:stability} to 
		obtain estimates for $\bm u$ and then apply the triangle inequality to obtain estimates for $\E$.  In particular, the stability estimate \eqref{eq:stability0} still holds, the regularity estimate \eqref{eq:stability1} 
		holds if an additional term $\ka^{-1}\|\dive\f\|$ is added to its right-hand side, and the $\bH^m$ regularity 
		estimate \eqref{eq:Hmstab} holds if the domain $\Om$ is $C^{m+1}$ and an additional term $\ka^{-m-1} \big(  \|\dive\f\|_{m-1}  +\ka  \|\dive\f\|_{m-2} \big) $ is added to $C_{m-2,\f,\g}$.  
	\end{remark}

	\section{An auxiliary Maxwell elliptic problem and {elliptic projections}}\label{sc3}
	In this section, we first construct a truncation operator for any divergence-free function by truncating its $\curl\curl$ eigenfunction expansion. {Then using} this operator, we define an auxiliary problem and present key regularity estimates. {Finally, we introduce corresponding elliptic projections and derive their error estimates. The results obtained in this section} are essential for the subsequent preasymptotic analysis.
	
	\subsection{A truncation operator}
	Recall that $\bX=\bH(\curl; \Om)\cap \bH_0(\dive^0; \Om)$  
	can be compactly embedded into $\bm{H}(\dive  ^0;\Om)$(see, e.g., \cite{weber1980local,monk2003}). Let $\bX^{* }$ be the dual space of $\bX$ 
	(all bounded linear functionals) and define the bounded linear operator $\mathcal{A}: \bX \rightarrow \bX^{*}$ by
	\eq{ \inn{ \mathcal{A} \bu, \bv}_{\bX^{*}, \bX} := (\curl \bu, \curl \bv), \quad \forall \bu,\bv \in \bX. 
	} Obviously, $\mathcal{A}$ is self-adjoint and coercive  on $\bX$. By applying  \cite[Theorem~2.37 and Corollary~2.38]{mclean2000strongly}
	we have that: there is a function sequence  $\{\bph_j \}_{j=1}^{\infty} \in \bX$, 
	and a strictly positive sequence $\left\{\la_{j}\right\}_{j=1}^{ \infty}$  satisfies:
	
	\begin{enumerate}
		\item $ \inn{\mathcal{A} \bph_j, \bv}_{\bX^{*},\bX} = (  \curl \bph_j, \curl \bv) =\la_{j}\left(\bph_j, \bv\right) \quad \forall j \geq 1, \bv \in \bX$;
		
		\item $\left\{\bph_j\right\}_{j=1}^{\infty}$ constitutes a set of complete orthonormal basis of $\bm{H}(\dive^0;\Om)$;
		
		\item $0 < \la_{1} \leq \la_{2} \leq \la_{3} \leq \ldots$, and when $j \rightarrow \infty$, $\la_{j} \rightarrow \infty$;
		
		\item For any $\bu, \bv \in \bX$, there holds $(\curl \bu, \curl \bv) = \sum_{j=1}^{\infty}\la_{j}(\bu, \bph_{j})(\bph_{j},\bv)$.  
	\end{enumerate}
	Let $\mathcal{P}:  =\curl\curl - \nabla \dive $. It is easy to verify that $\bph_j$ is the unique weak solution in $\bX$  of the following boundary value problem:
	\begin{equation}
		\left\{\begin{aligned}
			\mathcal{P} \bph   & =\la_{j}\bph_j   & \text { in } \Om, \\
			\curl	\bph   \times \n   & = \bm{0}   & \text { on } \Ga, \\
			\bph  \cdot \n   & = \bm{0}   & \text { on } \Ga.
		\end{aligned}\right.	
		\label{eq:Pwjeq}
	\end{equation}
	Therefore, in addition, we have 
	\begin{enumerate}[resume]
		\item $ \mathcal{P} \bph_{j}    =\la_{j}\bph_{j}    \text { in } \Om$ and 	$\curl	\bph_{j}   \times \n  = \bm{0}    \text { on } \Ga $;
		
		\item $\bph_{j}\in \bH^{m}(\Om)$ if $\Om$ is $C^m$ for any  $m\ge 2$.
	\end{enumerate}
	In fact, by  \cite[\S4.5b]{costabel2010corner}, the system \eqref{eq:Pwjeq}  is elliptic in the sense of \cite[Definition~3.2.2]{costabel2010corner}. Then by using the elliptic regularity given in \cite[Theorem~3.4.5]{costabel2010corner} for problem \eqref{eq:Pwjeq}, we have   $\bph_{j}  \in \bH^{m}(\Om)$.

	To define the truncation operator, we introduce the constant 
	\eq{\label{eq:CT}
		\cs=6+49\ctr^4\la^2}
	where $\ctr$ is the constant in the trace inequality \eqref{eq:trace} and $\la$ is the impedance constant. We remark that this constant comes from the proof of Lemma~\ref{lem:bcoev} and it is clear that $\cs \simeq 1$.  
	Then the truncation operator $\Sm : \bm{H}(\dive^0;\Om) \rightarrow \bX$ is defined by simply truncating the eigenfunction expansion  as follows:  
	\eqn{
		\text{for any } \bv=\sum_{j=1}^{\infty}\bv_j \bph_{j}\in \bm{H}(\dive^0;\Om),\quad \text{let } \Sm  \bv=  \sum_{j=1}^{N} \bv_j\bph_{j},}
	where $N \in \N_1$ is taken such that $\la_{N}\leq \cs \ka^{2} < \la_{N+1}$. 
	Obviously  $ 0\leq (\Sm  \bv, \bv) =\|\Sm \bv\|^2 \leq \|\bv\|^2 $  and   $(\Sm  \bv,\w)=(\bv, \Sm \w)$ \; $ \forall \bv,\w \in \bm{H}(\dive^0;\Om)$.
	
	\begin{remark}\label{rem:S}
		The truncation  operator $\Sm$ is inspired by \cite{galkowski2024sharp,chaumont2024sharp}, wherein functional calculus tools are utilized to formulate the operator in a more generalized and abstract manner. In comparison, akin to the approach taken in the Helmholtz case as described in \cite{li2023higherorderpml}, we define the smoothing operator explicitly by truncating eigenfunction expansions.  This straightforward and explicit construction is critical for preasymptotic error estimation when applying higher-order EEM to  \eqref{eq:eq} with impedance boundary condition, and enabling readers to focus on the essential methodological distinctions between \cite{lu2024preasymptotic} and the current work.
	\end{remark}

	\begin{lemma}[Stability of $\Sm $]\label{lem:Sstab}
		{Given} $m \in \N_1 $, suppose $\Om$ is $ C^{ \max\{2,m\}}$.  Then we have 
		$$
		\|\Sm  \bv\|_{m,\ka} \ls   \ka^m  \|\bv\|  \quad \forall \bv \in \bm{H}(\dive^0;\Om).
		$$	
	\end{lemma}

	\begin{proof}
		Let  $\mu_{j}=\left(\bv, \bph_{j} \right)$,  then  $\bv=\sum_{j=1}^{\infty} \mu_{j} \bph_{j}$, $\Sm  v=  \sum_{j=1}^{N} \mu_{j} \bph_{j} $ 
		and  $\|\bv\|^{2}=$ $\sum_{j=1}^{\infty}\left|\mu_{j}\right|^{2}$.  Then for any  $\alpha \in \N$,  we have
		$$
		\mathcal{P}^{\alpha} \Sm  \bv=  \sum_{j=1}^{N} \mu_{j} \mathcal{P}^{\alpha} \bph_{j} =  \sum_{j=1}^{N} \mu_{j} \la_{j}^{\alpha} \bph_{j} \in \bX.
		$$
		Using the orthogonality of $\left\{\bph_{j} \right\}_{j=1}^{\infty}$ and the fact that $0 < \la_{1} \leq \ldots \leq \la_{N} \leq \cs \ka^2$,  we obtain 
		$$
		\left\|\mathcal{P}^{\alpha} \Sm  \bv\right\| \ls    \bigg(\sum_{j=1}^{N} \la_{j}^{2 \alpha}\left|\mu_{j}\right|^{2}\bigg)^{\frac{1}{2}} \ls  \ka^{2 \alpha }\|\bv\| .
		$$
		Note that  $\left(\curl  \bph_{i}, \curl  \bph_{j} \right) 
		= \la_i \delta_{ij} $ and $\mathcal{P}^{\alpha} \Sm  \bv \in \bX $, we get by \eqref{eq:emb1} that
		$$
		\|\mathcal{P}^{\alpha} \Sm  \bv \|_1 \ls   \bigg\|\sum_{j=1}^{N} \mu_{j} \la_{j}^{\alpha} \curl\bph_{j}  \bigg\|   \ls
		\bigg(\sum_{j=1}^{N} \la_{j}^{2 \alpha+1}\left|\mu_{j}\right|^{2}\bigg)^{\frac{1}{2}} \ls  \ka^{2 \alpha+1}\|\bv\|.
		$$
		Note that  $\Sm \bv, \mathcal{P}^{\alpha} \Sm  \bv \in \bX $,  using the above two basic estimates and  the elliptic regularity theory   \cite[Theorem~3.4.5]{costabel2010corner}, 
		we may get the following high-order stability estimate for any $\ell \in \N$ by induction:
		$$
		\begin{aligned}
			& \|\Sm  \bv\|_{2 \ell+1} \ls   \left\|\mathcal{P}  \Sm  v\right\|_{2 \ell-1} + \|\Sm  \bv\|_{1} \ls  \cdots \ls  \left\|\mathcal{P}^{\ell} \Sm  v\right\|_{1} + \cdots +  \|\Sm  \bv\|_{1} \ls
			\ka^{2 \ell+1}\|\bv\| \quad  2 \ell+1 \leq m,\\
			& \|\Sm  \bv\|_{2 \ell+2} \ls   \left\|\mathcal{P}  \Sm  v\right\|_{2\ell } + \|\Sm  \bv\|_{1} \ls  \cdots \ls  \left\|\mathcal{P}^{\ell+1} \Sm  v\right\|  + \cdots + \|\Sm  \bv\|_{1} \ls
			\ka^{2 \ell+2}\|\bv\|  \quad   2 \ell+2 \leq m.
		\end{aligned}
		$$
		This completes the proof of this lemma.
	\end{proof}

	\subsection{The auxiliary Maxwell elliptic problem}
	
	Given $ (\fp, \gp) \in \bL^2(\Om)\times \bL^2_t(\Ga)$, we consider the following auxiliary problem for solving the vector field $\w$: 
	\begin{equation}
		\left\{\begin{array}{rlrl}
			\curl\curl \w-\ka^{2}\w+\cs\ka^2\Sm \textsf{P}^{\perp}\w & =\fp  & \text { in } \Om, \\
			\curl \w \times  \n -\ii\ka\la\w_T & =\gp  & \text { on } \Ga,
		\end{array}\right. 
		\label{eq:auxProb}
	\end{equation}
	where $\cs$ is defined in \eqref{eq:CT}. 
	Obviously, the weak formulation is:  Find  $\w \in \V $,  such that 
	\eq{
		b(\w,  \bv) =(\fp, \bv) +\inn{\gp, \bv_T}   \quad \forall \bv \in \V, \label{eq:bVP}
	} 
	where the sesquilinear form   $b(\cdot,\cdot): \V \times \V \rightarrow \C$ is defined by 
	\begin{equation}
		b(\bu,   \bv) :=(\curl  \bu,   \curl  \bv) -\ka^{2}(\bu,   \bv) +\cs\ka^2(\Sm  \textsf{P}^{\perp}\bu, \textsf{P}^{\perp}  \bv) -\ii \ka \la\inn{\bu_T, \bv_T} .
		\label{eq:defbuv}
	\end{equation}
	%It is easy to verify that $b(\cdot,\cdot)$ is a sesquilinear form.
	The following lemma says that the  sesquilinear form $b(\cdot, \cdot)$ satisfies the  continuity and G\aa rding inequality on $\V$, and is coercive on the subspace $\V\cap \bm{H}_0(\dive^0;\Om)$. 
	\begin{lemma}\label{lem:bcoev}
		For any $\bu,\bv \in \V$, there hold
		\begin{align}
			\big|	b(\bu,\bv) \big|& \ls \ener{\bu}  \ener{\bv},  \label{eq:bCont} \\
			\big(\re-\im \big)	b(\bu,\bu) & \geq \ener{\bu}^2 -2 \ka^2\|  \bu \|^2, 		\label{eq:bgarding} \\
			\sup_{ {\bm 0}\neq \bv\in\V} \frac{\abs{	b( \bu, \bv)}}{\ener{\bv}} & \gtrsim \ener{\bu}.  \label{eq:inf-sup}
		\end{align} 	
		Particularly, the following coercivity holds on the subspace $\V\cap \bm{H}_0(\dive^0;\Om)$:
		\eq{\label{eq:bCoer}
			\abs{	b(\bu,\bu)} & \gtrsim \ener{\bu}^2      \quad \forall \bu \in \V\cap \bm{H}_0(\dive^0;\Om). 
		}
	\end{lemma}
	
	\begin{proof} The proofs of \eqref{eq:bCont} and \eqref{eq:bgarding} are straightforward and omitted.
		Next, we   prove \eqref{eq:inf-sup}.  For any $\bu \in \V$, let $\bu_1 = \textsf{P}^0 \bu$ and  $\bu_2 = \textsf{P}^{\perp}\bu$, thus $\bu =  \bu_1 + \bu_2$. To show \eqref{eq:inf-sup}, it suffices to prove that
		\eq{\label{eq:bu1u2}  \big(\re+\im \big)	b(\bu_1 + \bu_2,\bu_2 - \bu_1) \gtrsim  \ener{\bu_1 + \bu_2} \ener{\bu_2 - \bu_1}.
		}
		Using the fact  that $\curl \bu_1 ={\bm{0}}$, $(\bu_1,\bu_2)=0$, and the Young's inequality, we have
		%------------------------------------------------------------------------
		\eqn{&\quad \big(\re+\im \big)	b(\bu_1 + \bu_2,\bu_2 - \bu_1)\\
			%------------------------------------------------------------------------
			&=\|\curl\bu_2\|^2  -\ka^{2} \|\bu_2\|^2 +   \cs\ka^2(\Sm \bu_2,\bu_2)+\ka^{2} \| \bu_1 \|^2 + \ka \la \big( \|(\bu_1)_T\|^2_{\Ga} -  \|(\bu_2)_T\|^2_{\Ga}  \big) \\ 
			& \quad + 2\ka \la \im \big(\inn{(\bu_1)_T, (\bu_2)_T }\big)\\
			%------------------------------------------------------------------------
			&\ge \Big(   \frac{1}{2}\|\curl\bu_2\|^2   + \frac{1}{2}\ka^{2} \|\bu_2\|^2  + \ka^{2} \|\bu_1\|^2    + \frac12\ka\la  \|(\bu_1)_T\|^2_{\Ga} + \frac{1}{2} \ka \la   \|(\bu_2)_T\|^2_{\Ga}     \Big) \\
			& \quad   + \Big(  \frac{1}{2}  \|\curl\bu_2\|^2  -\frac{3}{2}\ka^{2} \| 
			\bu_2\|^2   +\cs\ka^2(\Sm \bu_2,\bu_2)    -\frac72 \ka \la   \|(\bu_2)_T\|^2_{\Ga}  \Big) \\
			%------------------------------------------------------------------------
			&=: T_1 +T_2.
		}
		It is obvious that 
		\eqn{T_1 \geq \frac{1}{2} \big(   \ener{\bu_1}^2 + \ener{\bu_2}^2  \big) \geq \frac{1}{2} \ener{\bu_1 + \bu_2} \ener{\bu_2 - \bu_1}.} Next we prove $T_2 \geq 0$. 
		By Remark~\ref{rem:proj}{\rm (c)}, \eqref{eq:emb1} and the trace inequality \eqref{eq:trace}, there holds
		\eq{\frac72\ka \la \|(\bu_2)_T\|^2_{\Ga}\leq \frac72\ka\la\ctr^2 \|\bu_2\| \|\curl \bu_2 \|\le \frac14\|\curl \bu_2\|^2+ \frac{49}{4} \ctr^4\la^2 \ka^2 \|\bu_2\|^2.
			\label{eq:ep1} }
		Therefore, we have 
		\eqn{
			T_2   \geq \frac14\|\curl \bu_2\|^2  - \frac14 \cs  \ka^{2} \|\bu_2\|^2  + \cs\ka^2(\Sm \bu_2,\bu_2).
		}
		Let  $\mu_{j}=\left(\bu_2,  \bph_{j} \right)$,  then $\bu_2=\sum_{j=1}^{\infty} \mu_{j} \bph_{j} $ (converges in $\bX$), 
		$\|\bu_2\|^{2}= \sum_{j=1}^{\infty}\left|\mu_{j}\right|^{2}$ and $\|\curl \bu_2\|^2 = \sum_{j=1}^{\infty}\la_j\left|\mu_{j}\right|^{2}$.  Noting that $0 < \la_{1} \leq \ldots \leq \la_{N} \leq \cs \ka^{2}  < \la_{N+1} \leq \cdots $, we have 
		\eqn{T_2 & \geq  \Big(     \frac{1}{4} \sum_{j=1}^{\infty} \la_{j}\left|\mu_{j}\right|^{2} -\frac14 \cs \ka^{2} \sum_{j=1}^{\infty}\left|\mu_{j}\right|^{2} + \cs \ka^{2} \sum_{j=1}^{N}\left|\mu_{j}\right|^{2}    \Big)  \ge 0,  
		}
		which completes the proof of \eqref{eq:inf-sup}. The proof of \eqref{eq:bCoer} follows directly from \eqref{eq:bu1u2} since $\bu_1={\bm{0}}$ for any $\bu\in \bm{H}_0(\dive^0;\Om)$.
		This completes the proof of this lemma.
	\end{proof}

	The well-posedness of  \eqref{eq:bVP} is a direct corollary of \eqref{eq:inf-sup} and the generalized Lax-Milgram theorem (see, e.g., \cite{nevcas1962methode}). 
	
	It is straightforward to observe that  the auxiliary problem mentioned above differs from  \eqref{Maxwell} by adding some smoothed, divergence-free component to the left-hand side of the original equation.   Compared to the original problem  (see Theorem~\ref{thm:stability}), the solution to the auxiliary problem exhibits improved behavior with respect to $\ka$-dependence. Specifically, we have the following lemma {whose proof can be obtained by using standard techniques for regularity estimates for Maxwell equations (see e.g. \cite{hiptmair2011stability,lu2018regularity}) and is postponed to Appendix~\ref{sc:appendixA}}.
	
	\begin{lemma}[Regularity for the auxiliary Maxwell elliptic problem]\label{lem:goodRegular}
		Suppose  $\Om$ is $C^{2}$,  $\fp\in\bH(\dive^0;\Om)$  and $\gp\in \bH^{-\frac{1}{2}}_t(\dive_{\Ga};\Ga)$.  Let $\w$ be the solution to \eqref{eq:auxProb}, then the following stability and regularity estimates hold:
		\eq{
			\|\curl \w \|  + \ka \|\w\|   &  \ls  \ka^{-1}\|\fp\|  + \ka^{-\frac{1}{2}} \|\gp\|_{\Ga}, \label{bSfgeqStabEnerg}  \\
			\|\w\|_{1}  & \ls   \ka^{-1}\|\fp\|  + \ka^{-\frac{1}{2}} \|\gp\|_{\Ga} + \ka^{-1}\| {\dive_\Ga} \gp\|_{-\frac{1}{2},\Ga}. \label{bSfgE1}
		}
		If $\gp \in \bH^{\frac{1}{2}}_t(\Ga)$, then
		\eq{			\|\curl\w\|_{1} & \ls \|\fp\| + \ka^{ \frac{1}{2}} \|\gp\|_{\Ga}+\| \gp\|_{\frac{1}{2},\Ga}.  \label{bSfgcurlE1}}
		Furthermore, if $\Om$ is $C^{m}$   and  $\fp \in \bH^{m  -2}(\dive^0;\Om)$, $\fp\cdot\n \in \bH^{m-\frac{3}{2}}(\Ga)$, $\gp \in \bH^{m-\frac{3}{2}}_t(\dive_{\Ga};\Ga)$ for some integer $m\geq 2$.  
		Then we have $\w \in \bH^{m}(\Om)$  and  
		\eq{
			\|\w\|_{m}  & \ls   \ka^{m-2} \tilde{C}_{m-2,\fp,\gp},  \label{bSfgEm}
		}
		where 
		\eqn{ 
			\tilde{C}_{m-2,\fp,\gp} &:=    \ka^{-m+2}\big( \|\fp\|_{m-2,\ka} +  \|\gp\|_{m-\frac{3}{2},\ka  {,\Ga}} \big) \\
			& \qquad	  + \ka^{\frac{1}{2}}  \|\gp\|_{\Ga}  + \ka^{-m+1} \big( \|\fp\cdot \n\|_{m-\frac{3}{2},\Ga} +    \| {\dive_\Ga} \gp\|_{m-\frac{3}{2},\Ga} \big).
		}	
	\end{lemma}

	\begin{remark} \label{rem:Sfstab}   
		{\rm (a)}  {If  $\Om  \in C^{m} $ for some integer $m \geq 2$, $\fp \in \bH^{m  -2}(\dive^0;\Om)$, $\fp\cdot\n =0$
			and $\gp =\bm{0}$  on $\Ga$,}  then we have $\|\w\|_{m} \ls  \|\fp\|_{m-2,\ka}$.

		{\rm (b)} Similar to the original continuous problem, if $\dive \fp \neq 0 $, the stability estimate \eqref{bSfgeqStabEnerg} still holds, the regularity estimate \eqref{bSfgE1}  and \eqref{bSfgcurlE1} 
		hold if an additional term $\ka^{-2}\|\dive\fp\|$ and $\ka^{-1}\|\dive \fp\| $ is added to the right-hand side of \eqref{bSfgE1}  and \eqref{bSfgcurlE1}, respectively. The $\bH^m$ regularity 
		estimate \eqref{bSfgEm} holds if the domain $\Om$ is $C^{m+1}$ and an additional term $\ka^{-m} \big(  \|\dive\fp\|_{m-1}  +\ka  \|\dive\fp\|_{m-2} \big) $ is added to $\tilde{C}_{m-2,\fp,\gp}$.  
	\end{remark}
	
	\subsection{Elliptic projections}			For any $\bu \in \V$, we define its  elliptic projection $\Ph^{\pm}\bu  \in \V_h$ by
	\begin{equation}
		b\left(\bu- \Ph^{+}\bu, \bv_{h}\right)=0, \;  \; b\left( \bv_{h}, \bu- \Ph^{-} \bu \right)=0    \quad \forall \bv_{h} \in \V_h.
		\label{eq:EP-orth}
	\end{equation}
	
	The following lemma provides error estimates  in the energy and $\bL^2$ norms  for the  elliptic projections under the condition that $\ka h$ is sufficiently small, ensuring that $\Ph^{\pm}\bu$ are well-defined. The proof, which follows the approach in \cite[Appendix~A]{lu2024preasymptotic} with some modifications, is included in Appendix~\ref{sc:appendixB} for the reader's convenience.
	
	\begin{lemma}\label{lem:Pherror} Suppose that $\Om$ is $C^{2}$. There exists a  constant $\tilde C_0>0$ such that if $\ka h\leq\tilde C_0$, then the following error estimates hold for any $\bu \in \V$ and $\bv_{h}\in \Vh $:
		\eq{
			\ener{\bu -\Ph^{\pm}\bu } & \ls      \ener{\bu-\bv_{h}}, \label{eq:Pherror_energy} \\
			\|\bu-\Ph^{\pm}\bu \|    &  \ls \|\bu-\bv_{h} \|  +h\ener{\bu-\bv_{h}}   + h^{\frac{1}{2}}\|(\bu -\bv_h)_T\|_{\Ga}, \label{eq:Pherror_l2}  \\ 
			\|(\bu -\Ph^{\pm}\bu)_T\|_{\Ga}    & \ls  \|(\bu -\bv_{h} )_T\|_{\Ga} + h^{-\frac{1}{2}} \|\bu-\bv_{h}\|  +h^{\frac{1}{2}}\ener{\bu-\bv_{h}}. \label{eq:boundPherror}  
		}
	\end{lemma}

	The next lemma  gives some negative norm error estimates for the elliptic projections, which will be used in the modified duality argument.
	\begin{lemma}\label{lem:SpherrEst}
		Suppose that $\Om $ is $ C^{p+1}$ 	and $\ka h  \leq \tilde C_0$, then the following estimate holds:
		\eq{ 
			\big\| \textsf{P}^{\perp}\left(\bu -\Ph^{\pm}\bu \right)\big\|_{-p+1,\ka}  \ls   h ^{p} \inf_{\bv_{h} \in \V_h}\ener{ \bu -\bv_{h}},
			%\left\|\Sm \textsf{P}^{\perp}\left(\bu -\Ph^{\pm}\bu \right)\right\|  \ls  \ka^{p-1}  h ^{p} \inf_{\bv_{h} \in \V_h}\ener{ \bu -\bv_{h}},
			\quad \forall \bu \in \V.  \label{eqSvphv1} 
		}
	\end{lemma}
	\begin{proof}
		We only prove for the estimate of $\big\| \textsf{P}^{\perp}\left(\bu -\Ph^{-}\bu \right)\big\|_{-p+1,\ka}$, the proof for the other estimate is similar. For any $\bv\in \bH^{p-1}(\Om)$, from Remark~\ref{rem:Sfstab} ({\rm{a}}), we know that there exists $\w \in \bH^{p+1}(\Om) \cap \bm{H}(\dive^0;\Om)$, such that
		$$
		b(\w, \bph )=\left( \textsf{P}^\perp\bv, \bph \right) \quad \forall \bph \in \V,
		$$
		and $\|\w\|_{p+1} \ls \|\textsf{P}^\perp\bv\|_{p-1,\ka}\ls \|\bv\|_{p-1,\ka}$. By using \eqref{eq:EP-orth} and \eqref{eq:interpener} we have
		$$
		\begin{aligned}
			\big|\big( \textsf{P}^{\perp}(\bu -\Ph^{-}\bu),\bv\big)\big|&=\big|\big( \textsf{P}^{\perp}(\bu -\Ph^{-}\bu),\textsf{P}^\perp\bv\big)\big|
			=\big|\big( \textsf{P}^\perp\bv, \bu -\Ph^{-}\bu\big)\big| \\
			& =b\left(\w,\bu -\Ph^{-}\bu \right)  =b\left(\bw-\Pih \bw,\bu  -\Ph^{-}\bu \right) \\
			& \ls  \ener{\bw- \Pih \bw }  \cdot  \ener{ \bu -\Ph^{-}\bu }  \ls   h ^{p}\|\bw\|_{p+1}  \ener{ \bu -\Ph^{-}\bu } \\
			& \ls   h ^{p}\|\bv\|_{p-1,\ka} \ener{ \bu -\Ph^{-}\bu},
		\end{aligned}
		$$
		which, together with  \eqref{eq:Pherror_energy}, completes the proof of the lemma. 
	\end{proof}
	\begin{remark}\label{rem:SPerr}
		This error estimate in the negative norm is quite similar to the standard negative norm error estimate of higher-order FEM for the elliptic problems (see e.g. \cite{brenner2008mathematical}).		
	\end{remark}

	\section{Preasymptotic error estimates for the EEM}\label{sc4}
 
	In this section, we first establish the $\ka$-explicit regularity of the solution to the Maxwell problem \eqref{Maxwell} by decomposing it into a  non-oscillatory  elliptic part and an oscillatory yet smooth part. Based on this decomposition, we then establish wave-number-explicit stability estimates for the $p$-th order EEM \eqref{eq:EEM} using the ``modified duality argument" \cite{ZhuWu2013II}. Finally, we use these stability estimates to derive preasymptotic error estimates for the EEM.

	\subsection{$\ka$-explicit regularity by decomposition  for the Maxwell problem} 
	The following theorem says that the solution  {$\E$ to the Maxwell problem \eqref{Maxwell}}
	can be decomposed into the sum of an elliptic part and a smooth part:  $\E=\Ee +\Ea$,
	where $\Ee$ is non-oscillatory in the sense that its $\bH^2$-norm  is uniformly bounded for all ${\ka}$ and $\Ea$
	is oscillatory but the $\bH^j$-bound of $\Ea$ is available for any integer $j \geq 0$ when the domain is sufficiently smooth.
	\begin{theorem}[Regularity decomposition]\label{thm:reg-dec}
		Suppose that  $\Om$ is $C^{m}$ for some integer $m \geq 2$ and  $  \f  \in \bH(\dive^0;\Om)$, $\g=\g_T\in \bH^{\frac12}(\Ga)$. Let $\E$ be the solution to  problem \eqref{Maxwell}, then there exists a splitting $\E=\Ee +\Ea$ such that $\Ee \in \bH^{1}(\Om)$ and $\Ea \in \bH^{m}(\Om) $ satisfying 
		\eq{  \ka^{1-m}  \|\Ea\|_{m}  +  \ka^2 \|\Ee\|+  \ka\|\Ee\|_{1} \ls  \|\f\| + \|\g\|_{\frac{1}{2},\Ga}. \label{eq:EeStab12} 
		}
		Furthermore, if   $\f\cdot \n \in \bH^{ \frac{1}{2}}(\Ga)$ and $\dive_{\Ga} \g \in \bH^{\frac12}(\Ga)$,  then $\Ee \in \bH^{2}(\Om)$ and
		\eq{     
			\|\Ee\|_{2} \ls  	\|\f\|    + \ka^{-1}\|\f\cdot \n\|_{\frac{1}{2},\Ga} 
			+\|\g\|_{\frac{1}{2},\Ga} + \ka^{-1} \| \dive_{\Ga} \g\|_{\frac{1}{2},\Ga}. \label{eq:EeStab}}
	\end{theorem}
	\begin{proof}
		We begin by introducing a boundary frequency splitting of $\g$, inspired by the construction in \cite[Section~6]{melenk2023wavenumber}.
		Let $\g = \g_T = (g_1,g_2,g_3)^T \in \bH^{\frac12}(\Ga)$.  
		For each component $g_i \in H^{1/2}(\Ga)$ $(i=1,2,3)$, we define its harmonic extension $\ext g_i \in H^{1}(\Om)$ as the  {weak} solution of
		\begin{equation}\label{eq:harmonic-extension}
			\begin{cases}
				\Delta (\ext g_i) = 0, & \text{in }\Om, \\
				\ext g_i|_{\Ga} = g_i, & \text{on }\Ga.
			\end{cases}
		\end{equation}
		By elliptic regularity (\cf \cite{mclean2000strongly}), the operator $\ext$ extends continuously as $
		\ext : H^{s-\frac12}(\Ga) \to H^{s}(\Om), 
		\;  1  \le s \le m,
		$
		and hence
		\begin{equation}\label{eq:ext-regularity}
			\|\ext g_i\|_{s} \ls \|g_i\|_{s - \frac12,\Ga}, 
			\qquad  1 \le s \le m.
		\end{equation} 
		
		Next, following \cite[Section~4]{melenk2011},  {using the} frequency filters $H_{\Om}$ and $L_{\Om}$   there, we have  the splitting
		\begin{equation}\label{eq:frequency-splitting}
			\ext g_i = H_{\Om} \ext g_i + L_{\Om} \ext g_i, 
			\qquad i=1,2,3.
		\end{equation}
		By Lemmas~4.2--4.3 of \cite{melenk2011}, the two components satisfy the bounds
		\begin{align}
			\|H_{\Om} \ext g_i\|_{s'} &\ls \ka^{s'-s} \|\ext g_i\|_{s}, 
			& \forall\, 0 \le s' \le s, \label{eq:Hfrequency-bounds}\\
			\|L_{\Om} \ext g_i\|_s &\ls \ka^{s-1} \|\ext g_i\|_1, 
			& \forall\, s \ge 1. \label{eq:Lfrequency-bounds}
		\end{align}
		
		Define
		\[
		\GE = (H_{\Om} \ext g_1, H_{\Om} \ext g_2, H_{\Om} \ext g_3)^T, 
		\qquad 
		\GA = (L_{\Om} \ext g_1, L_{\Om} \ext g_2, L_{\Om} \ext g_3)^T.
		\]
		We set $\gE = (\GE)_{T}$ and $\gA = (\GA)_{T}$.  
		Clearly, $\g = \gE + \gA$ on $\Ga$.  
		
		By the trace inequality together with \eqref{eq:Hfrequency-bounds}--\eqref{eq:Lfrequency-bounds} and \eqref{eq:ext-regularity}, we obtain
		\eq{  
			&\|\gE\|_{\Ga} 
			\ls \|\GE\|_{1}^{1/2}\|\GE\|^{1/2} {\ls \ka^{-\frac12}\sum_{i=1}^3\norm{\ext g_i}_1}
			\ls \ka^{-\frac12}\sum_{i=1}^3\norm{g_i}_{\frac12,\Gamma}  \ls \ka^{-\frac12}\|\g\|_{\frac12,\Ga};  \label{eq:gEL2}  \\
			&	\|\gE\|_{j+\frac12,\Ga}  \ls \|\GE\|_{j+1} 
			\ls \sum_{i=1}^3\norm{\ext  g_i}_{j+1 } \ls \|\g\|_{j+\frac12,\Ga}, 
			\qquad 0 \le j \le 1; \label{eq:gE-regularity}  \\
			& 	\|\gA\|_{j+\frac12,\Ga} \ls  \|\GA \|_{j+1} \ls \ka^j\sum_{i=1}^3\norm{ \ext  g_i}_{1}  \ls \ka^j \|\g\|_{\frac12,\Ga}, 
			\qquad 0 \le j \le m-1.   \label{eq:gA-regularity}
		}
		A triangle inequality and \eqref{eq:gA-regularity} yields that
		\eq{
			\|\dive_{\Ga} \gE\|_{ \frac12,\Ga}  
			%			\ls 	\|\dive_{\Ga} \g \|_{ \frac12,\Ga} + 	\|\dive_{\Ga} \gA\|_{\frac12,\Ga}
			\ls 	\|\dive_{\Ga} \g \|_{ \frac12,\Ga} + 	\| \gA\|_{\frac32,\Ga}
			\ls 	\|\dive_{\Ga} \g \|_{ \frac12,\Ga} +  \ka 	\| \g \|_{\frac12,\Ga}.    \label{eq:divegE}
		}
		
		We now turn to the regularity decomposition of the Maxwell solution $\E$  by following the idea proposed in \cite{li2023higherorderpml}.
		Let $\Ee \in \V$ be the solution of
		\[
		b(\Ee,\bv) = (\f,\bv) + \inn{\gE , \bv_T} \qquad \forall \bv \in \V.
		\]
		Then the estimates for $\Ee$ in \eqref{eq:EeStab12}--\eqref{eq:EeStab} follow from Lemma~\ref{lem:goodRegular} together with \eqref{eq:gEL2}--\eqref{eq:gE-regularity} and \eqref{eq:divegE}.
		
		Define $\Ea = \E - \Ee$.  
		It follows that $\Ea$ solves
		\[
		a(\Ea,\bv) 
		= \cs \ka^2\big( \Sm \textsf{P}^{\perp} \Ee, \bv\big) 
		+ \inn{\gA , \bv_T}, 
		\qquad \forall  \bv \in \V.
		\]
		By Theorem~\ref{thm:stability},  the stability of $\Sm$ and $\textsf{P}^{\perp}$, the estimate of $\Ee$  and  \eqref{eq:gA-regularity},  we deduce
		\begin{align}
			\ka^{-1}\|\Ea\|_{1} + \|\Ea\| 
			&\ls \cs\ka \|\Sm \textsf{P}^{\perp} \Ee\|  + \ka^{-1}\|\gA\|_{\Ga} + \ka^{-2}\|\gA\|_{\frac12,\Ga} 
			\ls \ka^{-1}(\|\f\| + \|\g\|_{\frac12,\Ga}), \notag \\
			\|\Ea\|_{m} 
			&\ls \cs \ka^{m+1} \Big( 
			\|\Sm \textsf{P}^{\perp} \Ee\| 
			+  \ka^{-m+1}\|\Sm \textsf{P}^{\perp} \Ee\|_{m-2,\ka}
			+ \ka^{-m}\|(\Sm \textsf{P}^{\perp} \Ee)\cdot\n\|_{m-\frac32,\Ga} \Big)  \notag \\
			&\quad + \ka^{m-1}\|\gA\|_{\Ga} + \|\gA\|_{m-\frac32,\ka,\Ga} + \ka^{-1}\|\dive_\Ga \gA\|_{m-\frac32,\ka, \Ga}  \notag \\
			&\ls \ka^{m-1}\big(\|\f\| + \|\g\|_{\frac12,\Ga}\big). \label{eq:beStab}
		\end{align}
		This completes the proof.
	\end{proof}
	
	\begin{remark}\label{rem:reg-dec}
		
		{\rm (a)} As discussed in \cite{galkowski2024sharp,li2023higherorderpml}, the above proof presents a more technically efficient method  compared to existing regularity splitting approaches, such as the Fourier technique \cite{melenk2010dtn,melenk2011,melenk2020wavenumber,melenk2023wavenumber} and iterative techniques found in \cite{chaumont2019general,chaumont2022frequency}.
		
		{\rm (b)}  
			Alternatively, one may employ  the boundary frequency filters in \cite[Eq. (6.16)]{melenk2023wavenumber} to obtain the same estimates but under an analytic boundary regularity assumption.
			The technique, however, does not extend to domains with merely $C^m$ boundaries,  since the construction involves the harmonic extension of scalar surface potentials and analytic continuation techniques, which cannot be applied under our smoothness assumption.
			
			{\rm (c)} When $\dive\f \neq 0$,  the estimates for  $\|\Ea\|_{m}$ and  $ \|\Ee\|$ in \eqref{eq:EeStab12} still hold.  The estimates for  $\ka \|\Ee\|_{1}$  in \eqref{eq:EeStab12} still holds if $\ka^{-1}\|\dive\f\|$ is added to the right-hand side of \eqref{eq:EeStab12}. And the estimate \eqref{eq:EeStab} remain valid if  $\ka^{-1}\|\dive\f\| +\ka^{-2}\|\dive\f\|_1 $ is added to  its right-hand side and  $\Om$ is $C^{3}$ (see Remark~\ref{rem:reg} {\rm (b)}).
		\end{remark}
		
		\subsection{Stability estimates of EEM}

		The following theorem gives the wave-number-explicit stability of the EEM solution.   
		
		\begin{theorem}\label{thm:EEMStability}
			Suppose that $\Om \in C^{p+1}$, $\f\in \bL^2(\Om)$  and $\g\in\bL^2_t(\Ga)$. Let $\E_h$	 be the solution to \eqref{eq:EEM}.  There exist a constant  $C_0>0$ independent of $\ka$ and $h$ such that if  $\ka^{2p+1} h^{2p}\le C_0$,  then
			\begin{equation}
				\|\curl\E_h\| +\ka \|\E_h\|+\ka \|{\Eht}\|_{\Ga} \ls \|\f\| +  \|\g\|_{\Ga}.
				\label{eq:EEMStability}
			\end{equation} 
			And as a consequence, the EEM \eqref{eq:EEM} is well-posed.
		\end{theorem}
		
		\begin{proof} The proof is divided into the following steps.
			
			{\it Step 1. Estimating $\|\curl\E_h\|$ and $\|\Eht\|_{\Ga}$ by $\|\E_h\|$.}
			By \cite[Theorem~4.1]{lu2024preasymptotic}, the following bounds for $\|\curl\E_h\|$ and $\|\Eht\|_{\Ga}$ hold: 
			\begin{align}
				\|\Eht\|_{\Ga} &\ls \|\E_h\|+\ka^{-1}\big(\|\f\|+\|\g\|_{\Ga}\big), \label{eq:stabbd}\\
				\|\curl\Eh\|  &\ls \ka \|\E_h\|+ \ka^{-\frac12}\big(\|\f\|+\|\g\|_{\Ga}\big). \label{eq:stabEner}
			\end{align}
			
			{\it Step 2. Decomposing $\|\E_h\|$.} According to  Lemma~\ref{lem:Vh0},
			there exists $\Eh^0\in \bH_{0}(\curl; \Om)\cap\V_h$ such that 
			\begin{equation}
				\|\Eh-\Eh^0\| +h\|\curl(\Eh-\Eh^0)\| \ls
				h^{\frac{1}{2}}\|\Eht\| _{ \Ga }.
				\label{eq:ehcerr}
			\end{equation} 
			By Lemma~\ref{lem:HD},    we have the following discrete Helmholtz decomposition
			for $\Eh^0$:
			\begin{equation}
				\Eh^0=\w_{h}^0+\nabla r_h^0,
				\label{eq:EhcHD}
			\end{equation}
			where $r_h^0\in U_h^0$ and $\w_{h}^0\in \V_{h}^0$ is discrete divergence free. 
			And there exists $\w^0\in \bH_{0}(\curl; \Om)\cap\bH(\dive^0; \Om)$ such that   $\curl\w^0=\curl\Eh^0$,
			and
			\begin{equation}
				\|\w_{h}^0-\w^0\| \ls h\|\curl\Eh^0\|\ls  h^{\frac{1}{2}}\|\Eht\|_{\Ga} +h\|\curl\Eh\|,
				\label{eq:stabwerr}
			\end{equation}
			where we have used \eqref{eq:ehcerr} to derive the last inequality. By Lemma~\ref{lem:embdding} and the inverse inequalities  $h^{-\frac12}\|\Eht\|_{\Ga},\|\curl\Eh\|\ls h^{-1}\|\Eh\|$, we have 
			\eq{\label{eq:w0}
				\|\w^0\|_1& \ls  h^{-1}\|\Eh\|.}
			In a manner analogous to the derivation in \cite[Step(2) of Theorem~4.1]{lu2024preasymptotic},  we have the following estimate of $\|\E_h\|$:
			\begin{equation}
				\|\E_h\| \ls \|\Eh-\Eh^0\|+\| \w_h^0-\w^0\|+ \frac{1}{\ka^2} \|\f\| 
				+h^{\frac{1}{2}}\|\Eht\|_{\Ga} +\|\w^0\|.
				\label{eq:stabEhbound}
			\end{equation}
			The first two terms on the right-hand side can be bounded using \eqref{eq:ehcerr} and \eqref{eq:stabwerr}. Next we estimate the last term in \eqref{eq:stabEhbound}.
			
			{\it Step 3. Estimating $\|\w^0\|$ using the modified duality argument.} 
			Consider  the dual problem:   
			\begin{align}
				\curl\curl \ps -\ka^2\ps&=\w^0 \quad \text{ in } \Om,   \label{eq:stabDP1-1}\\
				\curl \ps\times \n +\ii \ka\la\ps_T &=\bm{0}  \quad \text{ on } \Ga  \label{eq:stabDP1-2}.
			\end{align}
			It is easy to verify that $\ps $ satisfies the following variational formulation:
			\begin{equation}
				a ( \bv,\ps)=( \bv,\w^0) \quad \forall \bv \in \V,
				\label{eq:stabdualVP1}
			\end{equation}
			and by  Theorems~\ref{thm:stability} and \ref{thm:reg-dec} and \eqref{eq:w0}, there exists the decomposition $\ps=\Psie+\Psia$, and  the following estimates hold:
			\begingroup
			\allowdisplaybreaks	
			\begin{align}
				\|\curl  \ps \|_{1}&+\ka\|\curl  \ps \| +
				\ka\|\ps \|_{1}+\ka^{2}\|\ps \|  +\ka^{2}\|\ps_T \|_{\Ga}   \lesssim \ka \| \w^0\|,  \label{eq:stabDP1} \\
				\|\ps\|_{2} &\ls \ka \|\w^0\| +\ka^{-1}  \|\w^0\|_1\ls    \ka \|\w^0\| +\ka^{-1}h^{-1}\|\Eh\|, \label{eq:stabDP2} \\
				{\|\Psie\|_{2}} &\ls \|\w^0\| +\ka^{-1}  \|\w^0\|_1\ls    \|\w^0\| +\ka^{-1}h^{-1}\|\Eh\|, \label{eq:stabDPE2} \\
				{\|\Psia\|_{p+1}} &\ls \ka^p\|\w^0\|.  \label{eq:stabDPA}			
			\end{align} 
			{Since $\ka h\leq \tilde C_0$ can be guaranteed if  $\ka^{2p+1}h^{2p}$ is sufficiently small, from  Lemmas \ref{lem:interpEst} and \ref{lem:Pherror}--\ref{lem:SpherrEst}, and \eqref{eq:stabDPE2}--\eqref{eq:stabDPA}, we have					\begin{align}
					&\|\ps-\pps \|  + h^{\frac{1}{2}}\|(\ps-\pps )_T\|_{\Ga}+h^{1-p}\| \textsf{P}^{\perp}(\ps - \pps)\|_{-p+1,\ka}\notag\\
					&\ls\|\ps-\Pih\ps \|  + h^{\frac{1}{2}}\|(\ps-\Pih\ps )_T\|_{\Ga}+h\ener{ \ps - \Pih\ps}  \notag\\
					&\ls  h^{2} \| \Psie\|_{2}+ h^{p+1} \| \Psia\|_{p+1} \notag \\
					& \ls  h\big((h+ \ka^{p}h^{p}) \|\w^0\|  + \ka^{-1} \|\Eh\|\big).  \label{eq:infDP}  
				\end{align}  
				Using  \eqref{eq:stabdualVP1}--\eqref{eq:stabDP2}, \eqref{eq:infDP},  and Lemma~\ref{lem:Sstab}, } 
			we deduce that
			\begin{align}
				& \quad  \abs{	( \Eh,\w^0) }  =\abs{ a ( \Eh, \ps) }  \notag \\
				&=\abs{ a ( \Eh, \pps) +a ( \Eh, \ps-\pps)} \notag \\
				&=\abs{ (\f, \pps)+ \bra \g, (\pps)_T \ket +  a ( \Eh, \ps-\pps) }  \notag \\
				&=\abs{(\f, \ps)+ \bra \g, \ps_T \ket +  (\f, \pps - \ps )+ 
					\bra \g, (\pps - \ps )_T \ket - \cs\ka^2 \big(\Sm \textsf{P}^{\perp}\Eh,\textsf{P}^{\perp}(\ps-\pps)\big)}  \notag \\ 
				&\ls   {\|\f\|\|\ps\|+\|\g\|_{\Ga}\|\ps_T\|_{\Ga}+\|\f\|\|\ps-\pps\|  
					+ \|\g\|_{\Ga}\|(\ps-\pps)_T\|_{\Ga} }  \notag\\
				&\quad+ \ka^2 \|\Sm \textsf{P}^{\perp}\E_h\|_{p-1,\ka} \| \textsf{P}^{\perp}(\pps - \ps)\|_{-p+1,\ka}          \notag \\ 
				&\ls \ka^{-1}  \big(\|\f\|+\|\g\|_{\Ga}\big)\|\w^0\|        \notag \\ 
				&\quad 
				+   {\big((h+ \ka^{p}h^{p}) \|\w^0\|  + \ka^{-1} \|\Eh\| \big)\big(h \|\f\|+h^\frac12\|\g\|_{\Ga}+\ka^{p+1}  h ^{p}   \|\E_h\|\big) }       \notag \\ 
				&\ls  \ka^{-1}\big(\|\f\|+\|\g\|_{\Ga}\big)\|\w^0\|+   (\ka^{p+1}h^{p+1}+\ka^{2p+1}h^{2p} )\|\w^0\|\|\Eh\|  + \ka^{p}h^p  \|\Eh\|^2    \notag \\ 
				&\quad +   \big( \ka^{-1}h  \|\f\|+ \ka^{-1}h^\frac{1}{2}  \|\g\|_{\Ga}\big)\|\E_h\|.
				\label{eq:stabehw}
			\end{align}
			\endgroup
			Note that $(\w^0,\nabla r_h^0)=0$, we have
			\eqn{\|\w^0\|^2 & =(\w^0+\Eh-\Eh +\w_{h}^0 +\nabla r_h^0 -\w_{h}^0, \w^0) \notag \\
				& = (\w^0-\w_{h}^0, \w^0)  +(\Eh^0-\Eh, \w^0) +(\Eh, \w^0),      }
			which together with \eqref{eq:stabehw} implies that
			\eqn{
				\|\w^0\|^2&\ls\|\w^0-\w_{h}^0\|^2+ \|\Eh^0-\Eh\|^2+\ka^{-1}\big(\|\f\|+\|\g\|_{\Ga}\big)\|\w^0\|+  (\ka^{p+1}h^{p+1} +\ka^{2p+1}h^{2p})\|\E_h\| \|\w^0\|             \notag \\ 
				&\quad +    \big( \ka^{-1}h  \|\f\|+ \ka^{-1}h^\frac{1}{2}  \|\g\|_{\Ga}\big)\|\E_h\| + \ka^p h^p \|\E_h\|^2.  }
			Therefore, by the Young's inequality we have
			\eq{\label{stabw0wEstimate}
				\|\w^0\|&\ls\|\w^0-\w_{h}^0\|+ \|\Eh^0-\Eh\|+\ka^{-1}\big(\|\f\|+\|\g\|_{\Ga}\big)+  C(\ka,h) \|\E_h\|,}
			where $C(\ka,h) =\big(\ka^{2p+1}h^{2p}+  (\ka h)^{p+1} + (\ka  h )^{\frac{p}{2}} +h + h^\frac{1}{2}  \big)$.
			
			{\it Step 4.  Summing up.} By plugging \eqref{stabw0wEstimate} into \eqref{eq:stabEhbound} and using \eqref{eq:ehcerr} and \eqref{eq:stabwerr}, we obtain
			\eqn{
				\|\E_h\| &\ls h^{\frac{1}{2}}\|\Eht\|_{\Ga}+h\|\curl\Eh\|+ \ka^{-1}\big(\|\f\|+\|\g\|_{\Ga}\big) +C(\ka,h)  \|\E_h\|, }
			which, together with \eqref{eq:stabbd}--\eqref{eq:stabEner} gives
			\begin{align*}
				\|\E_h\| & \ls \big(\ka^{-1}+\ka^{-1}h^\frac12+\ka^{-\frac12}h\big)\big(\|\f\|+\|\g\|_{\Ga}\big)+ \big(   C(\ka,h)  + \ka h \big) \|\E_h\|.
			\end{align*} 
			When $\ka h$ is small, $\big(   C(\ka,h)  + \ka h \big) \ls  \ka^{2p+1}h^{2p}+(\ka h)^\frac12$.  Therefore, there exists a constant $C_0>0$  such that if  $\ka^{2p+1}h^{2p}\le C_0$,  then
			\begin{align}
				\|\E_h\| \ls \ka^{-1}\big(\|\f\|+\|\g\|_{\Ga}\big).
				\label{eq:stabl2}
			\end{align} 
			Then \eqref{eq:EEMStability} follows from \eqref{eq:stabl2}, \eqref{eq:stabbd} and\eqref{eq:stabEner}. This completes the proof of the theorem.	
		\end{proof}
		
		\begin{remark}\label{rem:EEMwellposed}
			{\rm (a)} This stability estimate of the  EEM  solution is  of the same  order as
			those of the continuous solution (\cf \eqref{eq:stability0} in Lemma~\ref{thm:stability})  and still holds when $\dive\f\neq 0$.
			
			{\rm (b)} When $\ka^{2p+1}h^{2p}$ is large, the well-posedness of EEM is still open.
			
			{\rm (c)} The modified duality argument in Step 3 differs from the standard duality argument (or Aubin-Nitsche trick)  in that it replaces the interpolation of the solution to the dual problem with its Maxwell  elliptic projection.    The modified duality argument was first introduced to derive preasymptotic error estimates \cite{ZhuWu2013II,DuWu2015} and preasymptotic stability estimates \cite{Wu2018jssx} for FEM and  the  continuous interior penalty (CIP)-FEM in the context of the Helmholtz equation with large wave numbers.
		\end{remark}

		\subsection{Preasymptotic error estimates of EEM}

		The following Theorem gives the preasymptotic error bounds for the high-order EEM applied to \eqref{Maxwell} in both the energy norm and the $\ka$-scaled $\bL^2$ norm, which is the main result of this paper.
		
		\begin{theorem}\label{thm:ErrorResult}
			Under the conditions of Theorem~\ref{thm:EEMStability}, let $\E$ denote the solution to problem \eqref{Maxwell}. Then, there exist a constant $C_0>0$ independent of $\ka$ and $h$, such that if     $\ka^{2p+1} h^{2p} \leq C_0$, the following preasymptotic error estimates between $\E$ and the EEM solution $\E_h$ hold:
			\eq{\ener{\E-\Eh} &\ls   \big( 1 +  \ka^{p+1}h^{p} \big) \inf_{\bv_{h} \in \V_h}\ener{\E -\bv_h},	\label{eq:energyResultInf}  \\
				\ka\|\E-\Eh\|  &\ls \inf_{\bv_{h} \in \V_h} \big((   \ka h +  \ka^{p+1}h^{p} ) \ener{\E -\bv_h} +  \ka\|\E -\bv_h\| +  \ka h^{\frac{1}{2}}\|(\E -\bv_h)_T\|_{\Ga}\big). \label{eq:L2ResultInf}
			}
			Furthermore, if  $ \f \in \bm{H}^{p-1}(\dive^0;\Om)$,  $\f\cdot \n \in\bH^{p-\frac{1}{2}}(\Ga)$  and $\g\in \bH^{p-\frac{1}{2}}_t(\dive_{\Ga};\Ga)$, then
			\eq{ \ener{\E-\Eh} &\ls   \big( \ka^{p}h^p +  \ka^{2p+1}h^{2p} \big) C_{p-1,\f,\g},	\label{eq:energyResult}\\
				\ka \|\E-\Eh\|  &\ls  \big( (\ka h)^{p+1} +  \ka^{2p+1}h^{2p}  \big) C_{p-1,\f,\g} , \label{eq:L2Result}
			}
			where $C_{p-1,\f,\g}$ is defined as in Theorem~\ref{thm:stability} with $m=p+1$.
		\end{theorem}
		\begin{proof}
			Denote by $\bbeta =\E-\PEp$ and ${\bxi}_h=\Eh - \PEp$. Then we have $ \E-\Eh =\bbeta-{\bxi}_h$ and 
			\begin{align*} 
				a ({\bxi}_h,\bv_{h}) &= a (\bbeta, \bv_{h})  =  - \cs\ka^2 (\Sm \textsf{P}^{\perp} \bbeta,  \bv_{h})\quad {\forall \bv_h\in\V_h.}
			\end{align*} 
			Without loss of generality, suppose $\|\Sm \textsf{P}^{\perp} \bbeta\|\neq 0$.			From Lemmas~\ref{lem:Sstab} and \ref{lem:SpherrEst}, we have
			\begin{equation*}
				\|\Sm \textsf{P}^{\perp} \bbeta\| =\frac{( \textsf{P}^{\perp} \bbeta,\Sm\Sm\textsf{P}^{\perp} \bbeta)}{\|\Sm\textsf{P}^{\perp} \bbeta\|}\ls \frac{\|\textsf{P}^{\perp} \bbeta\|_{-p+1,\ka}\|\Sm\Sm\textsf{P}^{\perp} \bbeta\|_{p-1,\ka}}{\|\Sm\textsf{P}^{\perp} \bbeta\|}   \ls \ka^{p-1}h^{p} \inf_{\bv_{h} \in \V_h}\ener{\E -\bv_h}.
				%\label{eq:PEenerErr}
			\end{equation*}
			By using Theorem~ \ref{thm:EEMStability}  {(with $\bm f=-\cs\ka^2\Sm \textsf{P}^{\perp} \bbeta $ and $\g= \bm{0}$)} and Lemma~\ref{lem:SpherrEst}, we obtain
			\begin{align*} 
				\ener{\bxi_h}+ \ka\|{\bxi}_h\|  & \ls \ka^2 \| \Sm \textsf{P}^{\perp} \bbeta\|    \ls   \ka^{p+1}h^{p} \inf_{\bv_{h} \in \V_h}\ener{\E -\bv_h}.   
			\end{align*} 	 
			Combining Lemma~\ref{lem:Pherror} and the above estimates  and using the triangle  inequality, we obtain \eqref{eq:energyResultInf}--\eqref{eq:L2ResultInf}.
			The estimates \eqref{eq:energyResult}--\eqref{eq:L2Result} can be obtained by Lemma~\ref{lem:interpEst}, Theorem~\ref{thm:stability} and the fact that $\ka h \ls 1$. This completes the proof of this theorem.
		\end{proof}	
		
		We make some remarks on this theorem.
		\begin{remark}\label{rem:PreErr}
			{\rm (a)}  The significances of Theorem~\ref{thm:ErrorResult} include:
			\begin{itemize}
				\item It generalizes the preasymptotic error estimates of the linear EEM in \cite{lu2024preasymptotic} for problem \eqref{Maxwell} to the higher-order cases.
				\item It  generalizes the preasymptotic error estimates  of the $h$-FEM of arbitrary order $p$ applied to the Helmholtz problem \cite{DuWu2015} to the EEM for the time-harmonic Maxwell equations. 
				\item  Using the second-type EEM, it generalizes the preasymptotic  error estimates of problem 
				\eqref{eq:eq} with PEC  boundary conditions in \cite{chaumont2024sharp}, to the impedance boundary conditions. 
				\item It reveals the advantages of higher-order EEM  applied to the high frequency Maxwell problem in the preasymptotic regime. 
				In contrast, the theories presented in  \cite{melenk2020wavenumber,melenk2023wavenumber} primarily focus on the quasi-optimality of $hp$-EEM, i.e., in the asymptotic regime.
			\end{itemize}
			
			{\rm (b)} When $\dive\f \neq 0$,  the estimates in \eqref{eq:energyResult}--\eqref{eq:L2Result} remain valid  by adding additional term to  $C_{p-1,\f,\g}$ for domains $\Om$ with $C^{p+2}$ regularity (see Remarks~\ref{rem:reg} and \ref{rem:Sfstab}).
			
			{\rm (c)} The tailor-made truncation operator $\Sm$ is crucial for weakening the mesh condition (see Lemma~\ref{lem:SpherrEst}). If the   elliptic projections were replaced by those given in  \cite[(3.8)-(3.9)]{lu2024preasymptotic}  instead of \eqref{eq:defbuv} and \eqref{eq:EP-orth}, the best  result we could currently achieve  would be a preasymptotic error estimate, under the condition that $\ka^{p+2}h^{p+1}$ is sufficiently small.
			
			{\rm (d)} For problems with large wave number, a discrete solution of reasonable accuracy requires the pollution error $\ka^{2p+1}h^{2p}$ to be small enough (see, e.g. \cite{ainsworth2003dispersive,ainsworth2004dispersive,ihlenburg1997finite}). From this point of
			view, our mesh condition $\ka^{2p+1}h^{2p}\leq C_0$ is quite practical.	
		\end{remark}

		\section{Numerical example}\label{sc5}
		
		In this section, we simulate the following three-dimensional time-harmonic Maxwell  problem to illustrate our theoretical results:
		\begin{equation*}
			\begin{cases}
				\curl\curl\E-\ka^2\E  = \f, \quad &{\rm in}\quad \Om:={(0,1)}^3,\\
				\curl \E \times  \n -\ii\ka \E_T=\g,   \quad &{\rm on}\quad \Ga := \pa  \Omega,  \\
			\end{cases}
		\end{equation*}
		and $\f$ and $\g$ are so chosen  that the
		exact solution is  
		$\E= \big( \sin(\ka y)J_0(\ka r), \cos(\ka z )J_0(\ka r), \ii \ka J_0(\ka r) \big)^T$, where $r=(x^2 +y^2+z^2)^{\frac{1}{2}}$ and $J_0$ is the Bessel function of the first kind. Given a positive integer $M>0$,
		the computational domain is discretized first by dividing it into small cubes with side-length $h_0:=1/M$, then each small cube is further divided in to six tetrahedrons (see Figure~\ref{fig:mesh}). A simple calculation yields that the number of degrees of freedom of the linear system \eqref{eq:EEM} is
		$
		\text{DOF} := M(p+1)(3M^2p^2+3M^2p+M^2+6Mp+3M+3).
		$ 
		{As in} \cite[Section~10]{melenk2023wavenumber}, we denote by $N_{\la}$ the number of degrees of freedom per wavelength, which can be calculated by
		$$
		N_{\la} : = \frac{2 \pi \text{ DOF}^{\frac{1}{3}}}{\ka}.
		$$
		We plot the relative errors in the norm $\energy{\cdot}:=(\|\curl \cdot\|^2 + \ka^2\|\cdot\|^2)^{\frac{1}{2}}$ of EE solutions and EE interpolants for $p=1,2,3$ and $M=1,2,\cdots$  with fixed
		$N_{\la} =10$ in one figure (see Figure~\ref{fig:kheper10dofJ0}). One can see that the relative errors  of the EE solutions fit those of the corresponding EE interpolations when $\ka<10$, implying the absence of pollution errors for small wave numbers.  The relative error  of the linear EEM ($p=1$) grows  approximately  linearly with the increase of $\ka$ when $\ka>20$ and exceeds $100\%$ when $\ka >100$. For the quadratic EEM
		($p=2$), the pollution error also grows linearly with respect to $\ka$, but at a slower rate compared to the linear case.
		By closely examining the relative error curve of the cubic EEM ($p = 3$), we can observe that the pollution error also exists although it increases very slowly.
		
		\begin{figure*}[htbp]
			\centering
			\subfigure[The structured mesh ($M=6$).]{
				\begin{minipage}[t]{0.49\linewidth}  
					\centering
					\includegraphics[scale=0.4]{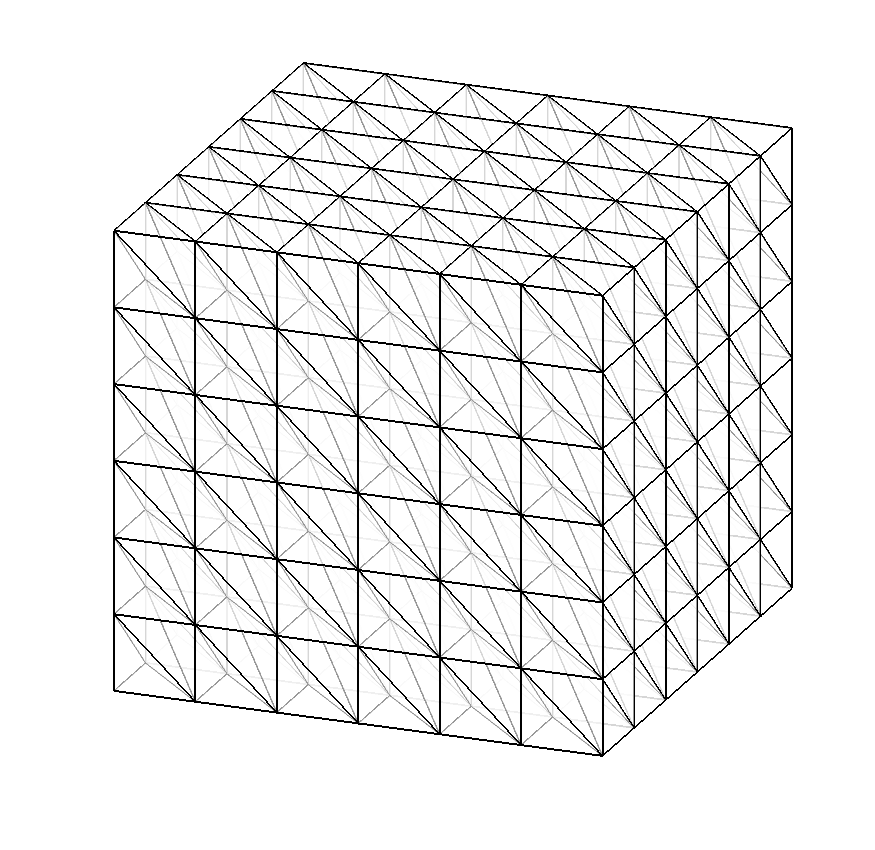}
					% 			\caption{}
					\label{fig:mesh}
				\end{minipage}%
			}%
			\subfigure[The relative  errors  in $\energy{\cdot}$ of the EE solutions  and the corresponding EE interpolations.]{
				\begin{minipage}[t]{0.49\linewidth}  
					\centering
					\includegraphics[scale=0.5]{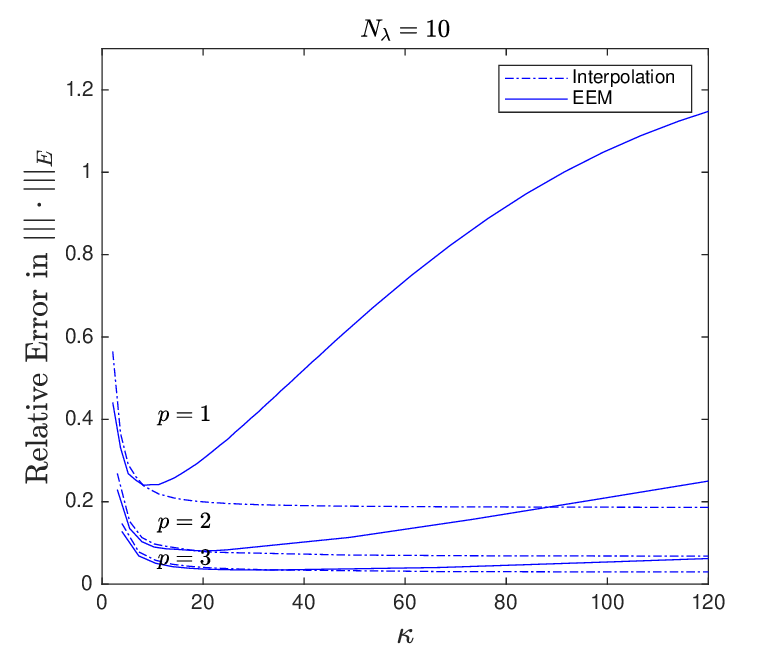}
					% 		\caption{}
					\label{fig:kheper10dofJ0}
				\end{minipage}%
			}%
			
			\centering  
			\caption{A sample mesh and relative error curve.}
			\label{fig:interp2}
		\end{figure*}
		Figures~\ref{fig:fixKEner} and \ref{fig:fixKL2} illustrate the relative  errors in $\energy{\cdot}$ and  $\bL^2$-norm  of the EE solutions and EE interpolations, for $\ka = 5$ and $50$, respectively.
		The figures clearly show that  when $\ka = 5$, the errors of the solutions close to those of the corresponding EE interpolations and decay with the optimal order, implying  that  pollution errors are negligible for small wave numbers.
		Conversely, for large $\ka$, the relative errors of the  EE  solutions decay slowly when $N_{\la}$ is small and quasi-optimal convergence is reached when $N_{\la}$ is sufficiently large. These behaviors vividly reveal the existence of pollution error in the fixed-order EEM. And it is also worth noting 
		that the higher-order EEM effectively alleviates the pollution effect, which is expected in view of Theorem~\ref{thm:ErrorResult}. 
		
		\begin{figure}[h]
			\centering
			\includegraphics[width=0.99\textwidth,scale=1,trim={2cm 0 2cm 0}]{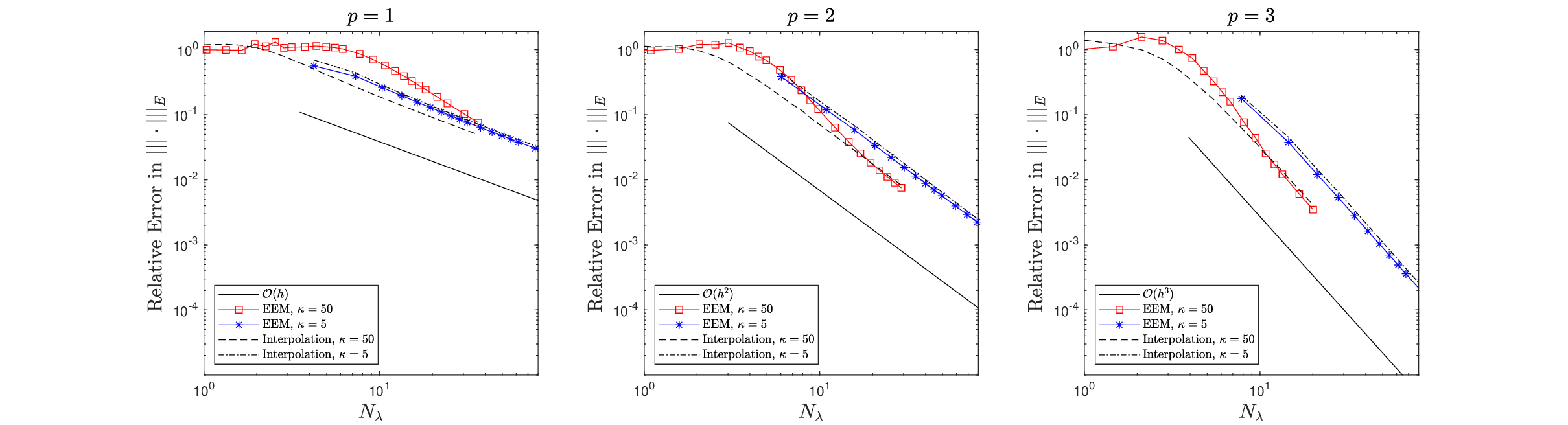}\\
			\caption{Log-log plots of  the relative errors in $\energy{\cdot}$ of  the   EE solution  versus  $N_{\la}$ with $p= 1,2,3 $, 
				and for $\ka  = 5$  and $50 $, respectively. 
				\label{fig:fixKEner}
			}
		\end{figure}
		\begin{figure}[h]
			\centering
			\includegraphics[width=0.99\textwidth,scale=1,trim={2cm 0 2cm 0}]{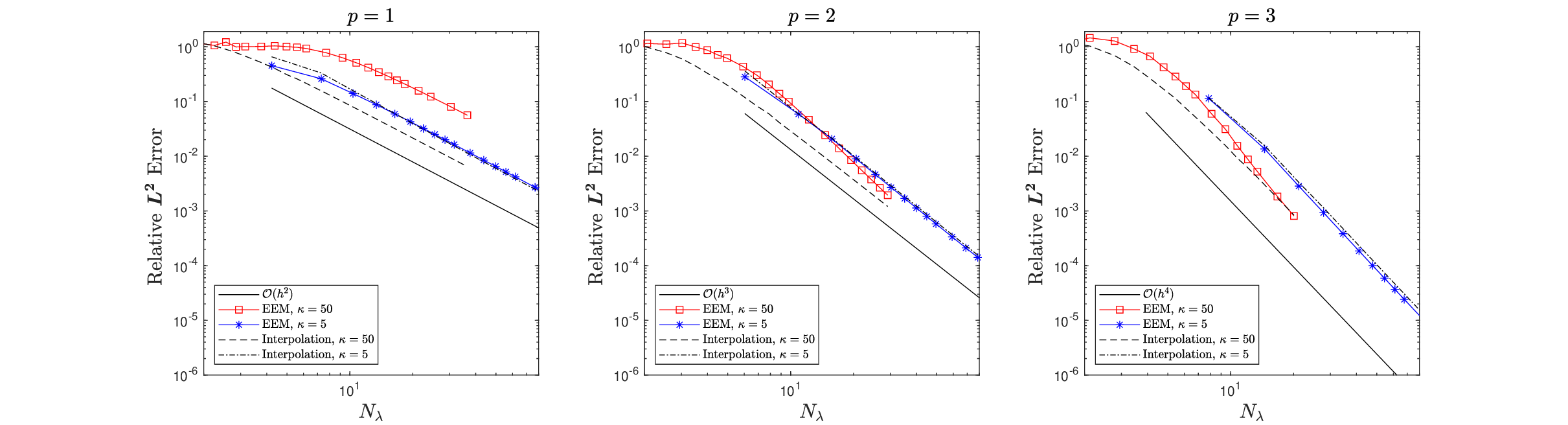}\\
			\caption{Log-log plots of the relative $\bL^2$ errors of  the  EE solution  versus $N_{\la}$ with $p= 1,2,3 $, 
				and	for $\ka = 5$  and $50 $, respectively. 
				\label{fig:fixKL2}
			}
		\end{figure}

		% -----------------------appendix --------------------------------
		
		\appendix
		
		\renewcommand{\theequation}{A.\arabic{equation}}
		\renewcommand{\thetheorem}{A.\arabic{theorem}}
		\renewcommand{\thelemma}{A.\arabic{lemma}}
		\renewcommand{\theremark}{A.\arabic{remark}}
		\renewcommand{\thefigure}{A.\arabic{figure}}
		\setcounter{equation}{0}
		\setcounter{theorem}{0}
		\setcounter{lemma}{0}
		\setcounter{remark}{0}
		\setcounter{figure}{0}
		
		\section{proof of Lemma~\ref{lem:goodRegular}}\label{sc:appendixA}
		
		{\it Proof of \eqref{bSfgeqStabEnerg}.}	
		By \eqref{eq:bVP}  and \eqref{eq:inf-sup}, there exists $ {\bm 0}\neq \bv\in \V$ such that
		\eqn{ \ener{\w} &\ls \frac{\abs{b(\bw,\bv)}}{\ener{\bv}} = \frac{\abs{(\fp, \bv) +\inn{\gp, \bv_T} }}{\ener{\bv}}\ls  \frac{ \|\fp\|\|\bv\|+ \|\gp\|_{\Ga}\|\bv_T \|_{\Ga}  }{\ener{\bv} }   \\
			&\ls \ka^{-1}\|\fp\| +\ka^{-\frac{1}{2}} \|\gp\|_{\Ga}, }
		which immediately deduces \eqref{bSfgeqStabEnerg}.
		
		{\it Proof of \eqref{bSfgE1}.}	
		To prove 	\eqref{bSfgE1}--\eqref{bSfgcurlE1}, we rewrite \eqref{eq:auxProb} as:
		\begin{equation}\label{eq:transferDefProb}
			\left\{\begin{array}{rlrl}
				\curl\curl \w-\ka^{2}\w&=\fp -\cs\ka^2\Sm \textsf{P}^{\perp} \w =: \fpp & \text { in } \Om, \\
				\curl \w \times  \n -\ii\ka\la\w_T&=\gp  & \text { on } \Ga,
			\end{array}\right. 
		\end{equation}
		then we follow  {the proof of} \cite[Theorem~3.2]{lu2018regularity} to derive \eqref{bSfgE1} and \eqref{bSfgcurlE1}.
		
		By Lemma~\ref{lem:proju}, there exist $\Bup^0 \in \bH_0(\dive^0;\Om)$  and  $\psi\in H^1$ such that $$\w= \Bup^0 + \nabla \psi.$$  
		{It is esay to verify that ${\curl} \Bup^0 =   \curl \w$, $\left\| \Bup^0 \right\|  \leq \|\w\|$,  $\|\nabla \psi \|  \leq \|\w\|$, and $\psi$ can be chosen according to the following elliptic equation in weak sense:
			\begin{equation*}  
				\Delta~ \psi = 0 \ \  \text{in} \ \Omega,\quad  \frac{\partial \psi}{\partial \n } =  \w\cdot \n\ \  \text{on}\ \ \Ga, \quad   \int_{\Omega} \psi\,  {\rm d}x = 0.
			\end{equation*} 
			By} virtue of  Lemma~\ref{lem:embdding},\begin{equation} \label{Aiota}
			\left\|\Bup^0 \right\|_{1} \ls \| \curl \w\|.
		\end{equation}
		On the other hand, using  Poincar\'e's inequality, we  get  that
		\begin{equation}\label{ApsiH1}
			\| \psi \|_{1} \ls  \|\nabla \psi \|  \ls  \|\w\|.
		\end{equation}
		Next, we consider the estimate of $\|\psi\|_2$. The boundary condition    can be rewritten as
		\begin{equation*} 
			\ii  \ka  \la \nabla_{\Ga} \psi = (   \curl \w ) \times \n -\ii \ka  \la \Bup^0_{T} - \gp,
		\end{equation*}
		where $\nabla_{\Ga}\psi$ is the tangential gradient of $\psi$, i.e., $\nabla_{\Ga} \psi = (\n \times \nabla \psi) \times \n {=(\nabla\psi)_T}$ {(see e.g.\cite{monk2003})}.
		Recall the formula $\dive_\Ga(\bm v\times\n)=\n\cdot\curl\bm v$ {(see e.g. \cite{monk2003})}, we have 
		\begin{equation*}  
			\dive_{\Ga} \big( (  \curl \w ) \times {\n }  \big)
			=  \n \cdot \big(  \curl (   \curl \w)    \big) = \n \cdot ( \fpp  +  \ka ^2   \w).
		\end{equation*}
		Hence it follows  that
		\begin{equation}\label{3-8}
			\begin{array}{ll}
				\left\| \dive _{\Ga}\big( ( \curl \w ) \times \n \big)   \right\|_{-\frac12, \Ga} 
				&\ls   \left\|\fpp  +  \ka ^2   \w \right\|  + \left\| \dive ( \fpp  +
				\ka ^2   \w) \right\|    \\
				&\ls  \|\fpp \|  +   \ka ^2 \|\w\|.
			\end{array}
		\end{equation}
		Meanwhile, by virtue of \eqref{Aiota},
		\begin{equation}\label{3-9}
			\left \|\ii \ka  \la \Bup^0_T \right\|_{\frac12, \Ga}
			\ls   \ka  \| \Bup^0\|_{1} \ls   \ka  \| \curl \w\| .
		\end{equation}
		Combining the estimates  \eqref{3-8}-\eqref{3-9}, we have
		\begin{equation}\label{3-9-add}
			\begin{array}{ll}
				\ka   \| \dive_{\Ga} \nabla_{\Ga} \psi   \|_{-\frac12, \Ga} 
				&\leq  \left \|  \dive_{\Ga} \big((    \curl \w) \times \n \big)\right\|_{-\frac12, \Ga} + \|\ii \ka  \la \Bup^0_T\|_{\frac12,\Ga} +  \|\dive_{\Ga} \gp\|_{-\frac{1}{2},\Ga}     \\ 
				&\ls   \| \fpp  \| 
				+    \ka^2 \|\w\|  +    \ka \| \curl \w\|  +  \|\dive_{\Ga} \gp\|_{-\frac{1}{2},\Ga}.
			\end{array}
		\end{equation}
		Hence, applying the elliptic regularity theorem for the Laplace-Beltrami operator on smooth surfaces (see e.g. \cite[Remark~1]{chen2024regularity}  {or \cite[\S1.3.b]{costabel2010corner}}), we have
		\begin{align*} 
			\ka  \| \psi \|_{\frac32, \Ga}
			&\ls   \ka   \left\| \dive_{\Ga} \nabla_{\Ga}\psi \right\|_{-\frac12, \Ga} +  
			\ka  \|\psi\|_{\Ga}    \\ 
			&\ls   \ka  \left\| \dive_{\Ga} \nabla_{\Ga} \psi  \right\|_{-\frac12, \Ga} +    \ka  \|\psi \|_{1} \\
			&\ls   \|\fpp  \| +   \ka ^2 \|\w\|   +  \ka \|  \curl \w\| +  \|\dive_{\Ga} \gp\|_{-\frac{1}{2},\Ga},
		\end{align*}
		where the last inequality is due to \eqref{ApsiH1} and \eqref{3-9-add}.
		Then according to the regularity theory for the Dirichlet problem of Laplace equation,
		\begin{equation*} 
			\ka  \|\psi\|_{2} \ls   \|\fpp  \| +   \ka ^2 \|\w\|  +  \ka \|  \curl \w\|    +   \|\dive_{\Ga} \gp\|_{-\frac{1}{2},\Ga},
		\end{equation*}
		which combined with \eqref{Aiota} yields
		\begin{equation}\label{3-12}
			\ka  \|\w\|_{1} \ls \|\fpp  \| +   \ka ^2 \|\w\| +  \ka \|  \curl \w\|  +  \|\dive_{\Ga} \gp\|_{-\frac{1}{2},\Ga}.
		\end{equation}
		By the definition of $\fpp$ and the stability of $\Sm$ and $\textsf{P}^{\perp}$, 
		we have
		\eqn{
			\|\w\|_{1} 
			& \ls \ka\|\w\| +\|\curl \w \| + \ka^{-1}\|  \fp -\cs\ka^2\Sm \textsf{P}^{\perp} \w \|  + \ka^{-1} \|\dive_{\Ga} \gp\|_{-\frac{1}{2},\Ga}  \\
			& \ls \ka^{-1}\|\fp\|  + \ka^{-\frac{1}{2}} \|\gp\|_{\Ga} + \ka^{-1}\|\dive \gp\|_{-\frac{1}{2},\Ga},
		}
		which completes the proof of \eqref{bSfgE1}. 
		
		{\it Proof of \eqref{bSfgcurlE1}.}	
		Note that $\curl\w$ satisfies 
		\begin{equation*} 
			\left\{ \begin{array}{l}
				\dive (   \curl \w) = 0,\ \ \  {\rm in}\ \Omega,\\[2mm]
				\curl (   \curl \w) = \fp +  \ka ^2  \w  -\cs\ka^2\Sm \textsf{P}^{\perp} \w ,\ \ \ {\rm in} \ \Omega, \\[2mm]
				\curl \w \times \n =  \ii  \ka  \la \w_T + \gp,\ \  \ {\rm on} \ \Ga.
			\end{array}
			\right.
		\end{equation*}
		Since $(\fp +  \ka ^2  \w  -\cs\ka^2\Sm \textsf{P}^{\perp} \w)\in \bL^2(\Om)$ and $(\ii  \ka  \la \w_T + \gp)\in \bH_{t}^{\frac{1}{2}}(\Ga)$, it follows from \cite[Theorem~2.2]{lu2018regularity} that $ \curl \w \in \bH^1(\Om)$ and
		\begin{align*} 
			\|   \curl \w\|_{1}  &
			\ls \|\fp +  \ka ^2  \w -\cs\ka^2\Sm \textsf{P}^{\perp} \w \|  + \| \ii  \ka  \la \w_T + \gp\|_{ \frac12, \Ga} +  \|  \curl \w\| \\ 
			&\ls  \|\fp \|  +  \ka ^2 \|\w\|  +  \|\gp\|_{ \frac12, \Ga} +  \ka  \|\w\|_{1}\\ 
			&\ls  \|\fp \|   + \ka^{\frac12} \|\gp\|_{\Ga} +  \|\g\|_{ \frac12, \Ga} ,
		\end{align*}
		which completes the proof of \eqref{bSfgcurlE1}.
		
		{\it Proof of \eqref{bSfgEm}.} 
		In order to prove  \eqref{bSfgEm}, we rewrite the  equations \eqref{eq:transferDefProb} as
		\begin{equation*} 
			\left\{\begin{array}{rlrl} 
				\curl \curl\w- \w  &= (\fp - \cs\ka^2\Sm \textsf{P}^{\perp} \w)+(\ka^2-1)\w   \qquad &{\rm in }\  \Omega,    \\
				\curl\w \times  \n -\ii \w_T &=\gp  +\ii(\ka-1) \w_T  \qquad  &{\rm on }\ \Ga.
			\end{array} 
			\right.
		\end{equation*}
		Again using the formula $\dive_\Ga(\bm v\times\n)=\n\cdot\curl\bm v$, we have from \eqref{eq:transferDefProb} that
		\eqn{
			\dive_\Ga\big(\ii\ka \w_T+\gp \big)&=\dive_\Ga(\curl \w \times  \n)=(\curl\curl \w)\cdot\n=\big(\ka^2 \w+(\fp - \cs\ka^2\Sm \textsf{P}^{\perp} \w)\big)\cdot\n,} which implies that
		\eqn{
			\dive_\Ga(\ii \w_T)=\ka \w\cdot\n+\ka^{-1}\big((\fp - \cs\ka^2\Sm \textsf{P}^{\perp} \w)\cdot\n-\dive_\Ga\gp \big).} Therefore, from \cite[Theorem~1]{chen2024regularity} and $(\Sm\textsf{P}^{\perp}\w)\cdot \n=0$, we have 
		\eqn{
			\|\w\|_2&\ls \big\|(\fp - \cs\ka^2\Sm \textsf{P}^{\perp} \w)+(\ka^2-1)\w\big\|+\big\|\gp  +\ii(\ka-1) \w_T\big\|_{\frac12,\Ga}\\
			&\quad+\big\|\big((\fp - \cs\ka^2\Sm \textsf{P}^{\perp} \w)+(\ka^2-1)\w\big)\cdot\n-\dive_\Ga\gp  -\ii(\ka-1)\dive_\Ga(  \w_T)\big\|_{\frac12,\Ga}\\
			&=\big\|(\fp - \cs\ka^2\Sm \textsf{P}^{\perp} \w)+(\ka^2-1)\w\big\|+\big\|\gp  +\ii(\ka-1) \w_T\big\|_{\frac12,\Ga}\\
			&\quad+\big\|\ka^{-1}( \fp \cdot\n-\dive_\Ga\gp )+(\ka-1)\w\cdot\n\big\|_{\frac12,\Ga}\\
			&\ls \|\fp\|+\ka^2\|\w\|+\|\gp \|_{\frac{1}{2},\Ga}+\ka\|\w\|_1+\ka^{-1}\|\fp\cdot\n\|_{\frac{1}{2},\Ga}+\ka^{-1}\|\dive_\Ga\gp\|_{ \frac{1}{2},  \Ga },
		}
		which, combined with \eqref{bSfgeqStabEnerg}--\eqref{bSfgE1}, proves \eqref{bSfgEm} for $m = 2$.
		Similar to the proof of \eqref{eq:Hmstab},  \eqref{bSfgEm} can be proved by  considering the higher-order regularity estimate of problem \eqref{eq:transferDefProb} and applying induction for $m$. We omit the details and conclude  the proof of the lemma.

		\renewcommand{\theequation}{B.\arabic{equation}}
		\renewcommand{\thetheorem}{B.\arabic{theorem}}
		\renewcommand{\thelemma}{B.\arabic{lemma}}
		\renewcommand{\thefigure}{B.\arabic{figure}}
		\setcounter{equation}{0}
		\setcounter{theorem}{0}
		\setcounter{lemma}{0}
		\setcounter{figure}{0}

		\section{proof of Lemma~\ref{lem:Pherror}}\label{sc:appendixB}
		Without loss of generality, we prove Lemma~\ref{lem:Pherror} only for $\tilde{{\bm u}}_h:=\boldsymbol{P}_h^+\bm u$. 
		Firstly, we bound  the  energy-norm error by the $\bL^2$-norm error. 
		\begin{lemma}\label{lem:energy}
			For any $\bu \in \V$, it holds that
			\begin{align}
				\ener{ {\bm u}-\tilde{\bm u}_h }   & \ls   \ener{ {\bm u}-{\bm v}_h} + \ka \|{\bm u}-\tilde{\bm u}_h \|  \quad \forall \bm v_h\in \V_h.   	\label{enerbyL2}
			\end{align}  
		\end{lemma}
		\begin{proof}
			By \eqref{eq:bgarding}  and \eqref{eq:EP-orth}, we have
			\eqn{ \ener{{\bm u}-\tilde{\bm u}_h}^2  & \leq \big(  \re -\im \big)  b({\bm u}-\tilde{\bm u}_h, {\bm u}-\tilde{\bm u}_h)  +2 \ka^2 \|{\bm u}-\tilde{\bm u}_h\|^2   \\
				&= \big(  \re -\im \big)  b({\bm u}-\tilde{\bm u}_h,{\bm u}-\bv_h )   +2 \ka^2 \|{\bm u}-\tilde{\bm u}_h\|^2 \\
				&\ls \ener{{\bm u}-\tilde{\bm u}_h}\ener{{\bm u}-\bv_h} + \ka^2 \|{\bm u}-\tilde{\bm u}_h\|^2.
			}
			Then \eqref{enerbyL2} follows by the Young's inequality.
		\end{proof}

		The following lemma bounds $\|({\bm u}-\tilde{\bm u}_h)_T\|_{\Ga}$ by $\|{\bm u}-\tilde{\bm u}_h\|$.
		
		\begin{lemma}\label{lem:bd}
			When $\ka h \ls 1$, there holds
			\begin{align}
				\|({\bm u}-\tilde{\bm u}_h)_T\|_{\Ga}  & \ls \|({\bm u}-{\bm v}_h)_T\|_{\Ga} + h^{-\frac{1}{2}}\|{\bm u}-{\bm v}_h\| +h^{\frac{1}{2}} \ener{ {\bm u}-{\bm v}_h} + \ka^{\frac{1}{2}}\|{\bm u}-\tilde{\bm u}_h\|   \quad \forall \bm v_h\in \V_h.   	\label{trace}
			\end{align}  
		\end{lemma}
		\begin{proof}
			Let ${\bm \Phi}_h:=\tilde{{\bm u}}_h-{\bm v}_h\in \V_h$. Similar  to Lemma \ref{lem:HD}, we have the following decompositions (see, e.g., \cite[Remark~3.46, Lemma~7.6]{monk2003}):
			\begin{align}
				\label{decom2}
				{\bm \Phi}_h=\nabla r+\w=\nabla r_h+\w_h,
			\end{align}
			where $r\in H^1(\Om)$, $\w\in \bH^1(\Om)\cap \bH_{0}(\dive^0;\Om)$, $r_h\in U_h$ and  $\w_h\in \V_h$ such that
			\begin{align}
				\label{wd}
				\|\w-\w_h\| \lesssim h\|{\curl}\ {\bm \Phi}_h\| \lesssim h \big( \ener{ {\bm u}-{\bm v}_h}+ \ka \|   {\bm u}-\tilde{{\bm u}}_h \|  \big),
			\end{align}
			where we have used Lemma~\ref{lem:energy} and the triangle inequality to derive the last inequality. 
			
			Next, we   establish  a relationship between $\|({\bm u}-\tilde{\bm u}_h)_T\|_{\Ga}$ and $\|\w\| $. Denote by 
			\begin{align*}
				d({\bm u},\bm v):=-\ka^2({\bm u},\bm v) + \cs\ka^2(\Sm \textsf{P}^{\perp}\bm u,\Sm\textsf{P}^{\perp}\bm v)-{\ii} \ka\la \langle {\bm u}_T,\bm v_T\rangle.
			\end{align*}
			From \eqref{eq:defbuv} and \eqref{eq:EP-orth}, we have
			\begin{align}\label{drh}
				d({\bm u}-\tilde{\bm u}_h,\na r_h)=b  ({\bm u}-\tilde{\bm u}_h,\na r_h)=0,
			\end{align}
			which implies
			\begin{align*}
				\im d({\bm u}-\tilde{\bm u}_h,{\bm u}-\tilde{\bm u}_h)=\im d({\bm u}-\tilde{\bm u}_h,{\bm u}-{\bm v}_h-\w_h),
			\end{align*}
			and hence from the Young's inequality and the $\bL^2$ stability of $\Sm$ and $\textsf{P}^{\perp}$ we obtain
			\begin{align*}
				\ka\la\|({\bm u}-\tilde{\bm u}_h)_T\|_{\Ga}^2\ls \ka\la\|({\bm u}-{\bm v}_h)_T\|_{\Ga}^2+\ka\la\|\w_{h,T} \|_{\Ga}^2+\ka^2\|{\bm u}-{\bm v}_h\| ^2+\ka^2\|\w_h\| ^2 +\ka^2 \|{\bm u}-\tilde{\bm u}_h\|^2.
			\end{align*}
			Therefore, by noting $\| \w_{h,T} \|_{\Ga}\ls h^{-\frac12}\|\w_h\|$, $\ka h \ls 1$, and using \eqref{wd}, we conclude that
			\begin{align}\label{lb}
				\|({\bm u}-\tilde{\bm u}_h)_T\|_{\Ga} &\ls  \|({\bm u}-{\bm v}_h)_T\|_{\Ga}+ \ka^{\frac12} \|{\bm u}-{\bm v}_h\| +\ka^\frac12    \|{\bm u}-\tilde{\bm u}_h\|  \notag \\
				&\quad  + h^{-\frac12}\|\w\| +h^\frac12\ener{ {\bm u}-{\bm v}_h }.
			\end{align}
			
			We now utilize the duality argument to estimate $\|\w\|$, first we begin by introducing
			the dual problem:
			\begin{alignat}{2}
				\label{D-equations}
				{\curl}\,{\curl}\,{\bm z}-\ka^2 {\bm z} +\cs\ka^2 \Sm {\textsf{P}^{\perp} \bm z}&=\w \qquad &&{\rm in }\ \Om,\\
				\label{D-boundary}{\curl}\,{\bm z}\times \boldsymbol{\nu}+{\ii}\ka \la {\bm z}_T&= \bm{0} \qquad &&{\rm on }\
				\Ga,
			\end{alignat}
			or in the variational form:
			\eq{\label{DPz}
				b (\bm v,\bm z)=(\bm v, \bm w)  \quad\forall \bm v\in \V.
			}
			Noting that $\w \cdot \boldsymbol{\nu} = 0$ on $\Ga$, we derive the following $\bH^2$ regularity estimate for $\bz$ from Lemma~\ref{lem:goodRegular}:
			\begin{align}
				\|{\bm z}\|_{2}&\ls \|\w\|.\label{D-regularity-H2}
			\end{align}
			From \eqref{DPz} and \eqref{eq:interpener}, we deduce that
			\begin{align}
				\notag
				({\bm u}-\tilde{\bm u}_h,\w)
				&=b  ({\bm u}-\tilde{\bm u}_h,{\bm z}-\Pih {\bm z})
				\lesssim   \ener{ {\bm u}-\tilde{\bm u}_h } \ener{ {\bm z}-\Pih {\bm z}  }\\
				&\ls   \ener{ {\bm u}-\tilde{\bm u}_h }  h\|\bm z\|_2\ls   h \big( \ener{ {\bm u}-{\bm v}_h}+ \ka \|   {\bm u}-\tilde{{\bm u}}_h \|  \big) \|\bm w\|. \label{Dpart4}
			\end{align}
			Since $\|\w\|^2=({\bm \Phi}_h,\w)=({\bm u}-{\bm v}_h,\w)-({\bm u}-\tilde{\bm u}_h,\w),$ we have from \eqref{Dpart4} that
			\begin{align}
				\label{w}
				\|\w\| \lesssim \|{\bm u}-{\bm v}_h\|+   h\big( \ener{ {\bm u}-{\bm v}_h}+ \ka \|   {\bm u}-\tilde{{\bm u}}_h \|  \big),
			\end{align}
			which together with \eqref{lb} completes the proof of this lemma. 
		\end{proof}
		
		The following lemma gives an estimate of $\|{\bm u}-\tilde{{\bm u}}_h\|$. 
		
		\begin{lemma}\label{lem:l2err}
			When $\ka h $ is sufficiently small, we have
			\begin{align}
				\label{error}
				\|{\bm u}-\tilde{{\bm u}}_h\|  \lesssim  \|{\bm u}-{\bm v}_h\| +h\ener{ {\bm u}-{\bm v}_h }+h^\frac12\|({\bm u}-{\bm v}_h)_T\|_{\Ga}  \quad \forall \bm v_h\in \V_h.
			\end{align}
		\end{lemma}
		\begin{proof} The idea is to convert the estimation of $\|{\bm u}-\tilde{{\bm u}}_h\| $ to that of $\|({\bm u}-\tilde{\bm u}_h)_T\|_{\Ga}$. Suppose $\ka h\ls 1$. For
			${\bm \Phi}_h=\tilde{{\bm u}}_h-{\bm v}_h$, according   to Lemma~\ref{lem:Vh0}, there exists ${\bm \Phi}_h^c\in \V_h^0$ such that
			\begin{align}
				\label{L2}
				\|{\bm \Phi}_h-{\bm \Phi}_h^c\| +h\|{\curl\,}({\bm \Phi}_h-{\bm \Phi}_h^c)\| &\lesssim h^{\frac{1}{2}}\| {\bm \Phi}_{h,T}\|_{\Ga}.
			\end{align}
			From Lemma~\ref{lem:HD} we have the following discrete Helmholtz decomposition for ${\bm \Phi}_h^c$:
			$${\bm \Phi}_h^c=\w_h^0+\nabla r_h^0,$$
			where $r_h^0\in U_h^0$ and $\w_h^0\in\V_h^0$ is discrete divergence-free. Moreover, there exists  $\w^0\in \bH_0({\curl})$ such that $\dive \w^0=0$, $\curl w^0=\curl {\bm \Phi}_h^c$ and
			\begin{align}
				\label{w1}
				\|\w_h^0-\w^0\|\lesssim h\|{\curl}\,\w_h^0\| =h\|{\curl}\,{\bm \Phi}_h^c\|.
			\end{align}
			From \eqref{eq:EP-orth}, we know that
			\begin{align}
				\label{dis-div}
				({\bm u}-\tilde{\bm u}_h,\nabla \phi_h^0)=0  \quad \forall \phi_h^0\in U_h^0.
			\end{align}
			Next, we  introduce the following dual problem:
			\begin{alignat}{2}
				\label{A-equations}
				{\curl}\,{\curl}\,{{\bm \Psi}}-\ka^2 {{\bm \Psi}} + \cs\ka^2{\Sm \textsf{P}^{\perp}{\bm \Psi}}&=\w^0  \quad &&{\rm in }\ \Om,\\
				\label{A-boundary}{\curl}\,{{\bm \Psi}}\times \boldsymbol{\nu}+{\ii}\ka \la {{\bm \Psi}}_T&= \bm{0} \quad  &&{\rm on }\  \Ga,
			\end{alignat}
			or in the variational form: 
			\eq{\label{DAz}
				b (\bm v,\bm \Psi)=(\bm v, \bm w^0)  \quad\forall \bm v\in \V.
			}
			By Lemma~\ref{lem:goodRegular}, we have the following  estimates:
			\begin{align}
				\label{A-regularity}
				\|{\bm \Psi}\|_{\bH^1(\curl)}&\lesssim \|\w^0\|,\\
				\label{A-regularity-H2}
				\|{\bm \Psi}\|_{2}&\ls \|\w^0\| +\ka^{-1}\|\w^0\|_{1}\lesssim \|\w^0\| +\ka^{-1}\|\curl\w^0\| \notag\\
				&=\|\w^0\| +\ka^{-1}\|{\curl}\,{\bm \Phi}_h^c\|.
			\end{align}
			Using \eqref{DAz}, we obtain
			\begin{align}
				&({\bm u}-\tilde{\bm u}_h,\w^0)
				\label{part4} =b ({\bm u}-\tilde{\bm u}_h,{\bm \Psi})=b ({\bm u}-\tilde{\bm u}_h,{\bm \Psi}-\Pih\bm \Psi).
			\end{align}
			Using \eqref{A-regularity}--\eqref{part4} and Lemma~\ref{lem:interpEst}, we conclude that
			\begin{align}\label{approx4}
				&\big|({\bm u}-\tilde{\bm u}_h,\w^0)\big|
				\lesssim \ener{ {\bm u}-\tilde{\bm u}_h } h \|{\bm \Psi}\|_{\bH^1(\curl)} +\big(\ka^2h^2\|{\bm u}-\tilde{\bm u}_h\|+\ka h^{\frac32}\|({\bm u}-\tilde{\bm u}_h)_T\|_{\Ga}\big)\|{\bm \Psi}\|_{2}\notag
				\\
				&\lesssim  h\ener{ {\bm u}-\tilde{\bm u}_h } \|\w^0\|+\big( h^2\ener{ {\bm u}-\tilde{\bm u}_h } +  h^{\frac32}\|({\bm u}-\tilde{\bm u}_h)_T\|_{\Ga}\big)\|{\curl}\,{\bm \Phi}_h^c\|.
			\end{align}
			From \eqref{dis-div} and the orthogonality between $\w^0$ and $\nabla r_h^0$, we
			may get
			\eqn{\|{\bm u}-\tilde{\bm u}_h\|^2+\|\w^0\| ^2&=({\bm u}-\tilde{\bm u}_h+\w^0,{\bm u}-\tilde{\bm u}_h+\w^0)-2\re({\bm u}-\tilde{\bm u}_h,\w^0)\\
				&=({\bm u}-\tilde{\bm u}_h+\w^0,{\bm u}-{\bm v}_h)-({\bm u}-\tilde{\bm u}_h+\w^0,{\bm \Phi}_h-{\bm \Phi}_h^c)\\
				&\quad-({\bm u}-\tilde{\bm u}_h+\w^0,\w_h^0-\w^0)-2\re({\bm u}-\tilde{\bm u}_h,\w^0), }
			which together with  (\ref{approx4}) and the Young's inequality, gives
			\begin{align*}
				\|{\bm u}-\tilde{\bm u}_h\|^2 +\|\w^0\| ^2&\lesssim 
				\|{\bm u}-{\bm v}_h\| ^2+\|{\bm \Phi}_h-{\bm \Phi}_h^c\| ^2+\|\w_h^0-\w^0\| ^2\\
				&\quad+ h^2\ener{ {\bm u}-\tilde{\bm u}_h }^2+  h \|({\bm u}-\tilde{\bm u}_h)_T\|_{\Ga}^2+h^2\|{\curl}\,{\bm \Phi}_h^c\| ^2.
			\end{align*}
			By Lemma~\ref{lem:energy}, (\ref{L2}) and (\ref{w1}), we have
			\begin{align}
				\notag
				\|{\bm u}-\tilde{\bm u}_h\|+\|\w^0\| &\lesssim \|{\bm u}-{\bm v}_h\| +h^{\frac{1}{2}}\| {\bm \Phi}_{h,T} \|_{\Ga}+h\|{\curl}\,{\bm \Phi}_h^c\| \\
				\label{Hw}
				&\quad+  h\big( \ener{ {\bm u}-{\bm v}_h}+ \ka \|   {\bm u}-\tilde{{\bm u}}_h \|  \big)+ h^\frac12\|({\bm u}-\tilde{\bm u}_h)_T\|_{\Ga}.
			\end{align}
			From (\ref{L2}) and Lemma~\ref{lem:energy} we have
			\begin{align*}
				h\|{\curl}\,{\bm \Phi}_h^c\| &\leq h\|{\curl}\,({\bm \Phi}_h-{\bm \Phi}_h^c)\| +h\|{\curl}\,{\bm \Phi}_h\| \lesssim h^{\frac{1}{2}}\| {\bm \Phi}_{h,T} \|_{\Ga}+ h\big( \ener{ {\bm u}-{\bm v}_h}+ \ka \|   {\bm u}-\tilde{{\bm u}}_h \|  \big).
			\end{align*}
			While using  the triangle inequality, we get
			\begin{align*}
				h^{\frac{1}{2}}\| {\bm \Phi}_{h,T} \|_{\Ga}&\leq h^{\frac{1}{2}}\big(\|({\bm u}-\tilde{\bm u}_h)_T\|_{\Ga}+ \|({\bm u}-{\bm v}_h)_T\|_{\Ga}\big).
			\end{align*}
			Inserting the above two inequalities into (\ref{Hw}), we obtain that the following estimate holds if  $\ka h$ is sufficiently small, 
			\begin{align}
				\label{L2tilde}
				\|{\bm u}-\tilde{{\bm u}}_h\| & \lesssim h^{\frac{1}{2}}\|({\bm u}-\tilde{\bm u}_h)_T\|_{\Ga}+h^{\frac{1}{2}}\|({\bm u}-{\bm v}_h)_T\|_{\Ga}+\|{\bm u}-{\bm v}_h\| + h\ener{ {\bm u}-{\bm v}_h }  \quad \forall \bm v_h\in\V_h,
			\end{align}
			which together with Lemma~\ref{lem:bd} completes the proof of the lemma.
		\end{proof}
		
		Finally, the proof of Lemma~\ref{lem:Pherror} follows by combining Lemmas~\ref{lem:energy}--\ref{lem:l2err} and using the fact that $\ka h$ is sufficiently small.\hfill \qedsymbol

		\bibliographystyle{abbrv} % plain, abbrv, siam, ...
		\bibliography{reference}

	\end{document}